\documentclass{amsart}
\usepackage{amssymb,amsmath, amsthm,latexsym}
\usepackage{mathtools,tikz-cd}
\usepackage{graphics}
\usepackage{psfrag}
\usepackage{amscd}
\usepackage{graphicx}
\newcommand{\cal}[1]{\mathcal{#1}}
\theoremstyle{plain}
\newtheorem{theo}{Theorem}

\newtheorem{lemma}{Lemma}[section]
\newtheorem{theorem}[lemma]{Theorem}  
\newtheorem{proposition}[lemma]{Proposition}
\newtheorem{corollary}[lemma]{Corollary}
\theoremstyle{definition}

\newtheorem{remark}[lemma]{Remark}
\let\egthree=\phi
\let\phi=\varphi
\let\varphi=\egthree

\begin{document}
\title{Spotted disk and sphere graphs II}
\author{Ursula Hamenst\"adt}
\thanks{Partially supported by ERC Grant ``Moduli'' and the Hausdorff Center Bonn}
\date{July 22, 2023}


\begin{abstract}
The disk graph of 
a handlebody $H$ of genus $g\geq 2$
with $m\geq 0$ marked points on the 
boundary is the graph whose vertices are isotopy classes of disks
disjoint from the marked points 
and where two vertices are connected by 
an edge of length one if they can be realized disjointly.
We show that for $m=2$ the disk graph contains
quasi-isometrically embedded copies of 
$\mathbb{R}^2$. Furthermore, the 
sphere graph of the
doubled handlebody of genus $g\geq 4$ 
with two marked points contains for every
$n\geq 1$ a quasi-isometrically embedded copy of $\mathbb{R}^n$.
\end{abstract}

\maketitle


\section{Introduction}

The \emph{curve graph} ${\cal C\cal G}$
of an oriented surface $S$ of genus $g\geq 0$
with $m\geq 0$ punctures and
$3g-3+m\geq 2$ is the graph whose vertices 
are isotopy classes of 
essential (that is, non-contractible and not homotopic 
into a puncture) simple
closed curves on $S$. Two such 
curves are connected by an edge of length one if and only
if they can be realized disjointly. 
The curve graph 
is a locally infinite hyperbolic geodesic metric space of 
infinite diameter \cite{MM99}.

A handlebody of genus $g\geq 1$  
is a compact three-dimensional manifold $H$  which can
be realized as a closed regular neighborhood in $\mathbb{R}^3$
of an embedded bouquet of $g$ circles. Its boundary
$\partial H$ is an oriented surface of genus $g$. We allow
that $\partial H$ is equipped with $m\geq 0$ marked points
(punctures) which we call \emph{spots} in the sequel.
The group ${\rm Map}(H)$ 
of all isotopy classes of orientation preserving
homeomorphisms of $H$ which fix each of the spots
is called the \emph{handlebody group} of $H$. 
The restriction of an element of ${\rm Map}(H)$ to the 
boundary $\partial H$ defines an embedding of 
${\rm Map}(H)$ into the mapping class group of $\partial H$, viewed as
a surface with punctures \cite{S77,Wa98}. 

An \emph{essential disk} in $H$ is a properly embedded
disk $(D,\partial D)\subset (H,\partial H)$ whose
boundary $\partial D$ 
is an essential simple closed curve in $\partial H$, viewed as a surface with 
punctures. An isotopy of 
such a disk is supposed to consist of such disks.

The \emph{disk graph} ${\cal D\cal G}$
of $H$ is the graph whose vertices
are isotopy classes of essential disks in $H$. 
Two such disks are connected by an edge 
of length one if and only if
they can be realized disjointly.
Thus by identifying a disk with its boundary circle, 
the disk graph is a subgraph of the curve graph of 
$\partial H$ which is invariant under the handlebody
group ${\rm Map}(H)$. 

If $H$ does not have spots, then it is known that the 
disk graph is a quasi-convex subgraph of ${\cal C\cal G}$.
This means that for any two points $D,E\in {\cal D\cal G}$,
any geodesic in ${\cal C\cal G}$ connecting $\partial D$ to 
$\partial E$ is contained in a uniformly bounded neighborhood of 
${\cal D\cal G}$. However, the inclusion 
${\cal D\cal G}\to {\cal C\cal G}$ is \emph{not} a 
quasi-isometric embedding \cite{MS10}. More precisely, 
there are uniformly quasi-isometrically embedded 
subgraphs of ${\cal D\cal G}$, so-called \emph{holes}, whose diameter in 
${\cal D\cal G}$ 
is arbitrarily large but whose diameter in the curve graph is uniformly 
bounded.

A metric space $X$ is said to have \emph{asymptotic dimension}
${\rm asdim}(X)\leq n$ if for every $R>0$ there exists a
covering of $X$ by uniformly bounded subsets of $X$ so that
any ball of radius $R$ intersects at most $n+1$ sets from the covering. 
The asymptotic dimension of a curve graph is finite \cite{BF08}
(see also \cite{BB19} for a quantitative statement). 

In \cite{MS10,H19a,H16,H19b,H12} the following is shown.

\begin{theo}\label{disk}
\begin{enumerate}
\item 
The disk graph of a handlebody of genus $g\geq 2$ without spots is
hyperbolic and has finite asymptotic dimension. 
\item The disk graph of a handlebody of genus $g\geq 2$ 
with a single spot on the boundary contains quasi-isometrically embedded
$\mathbb{R}^2$. In particular, it is not hyperbolic.
\end{enumerate}
\end{theo}

The mechanism for part (2) of Theorem \ref{disk} consists in taking 
advantage of specific holes for the disk graph of a handlebody without
spots which arise from representing the
handlebody as a (possibly non-oriented) $I$-bundle over a compact 
surface $F$ with connected boundary together with a point pushing 
construction about the boundary circle of $F$. 

The first main goal of this work
is to extend the second part of Theorem \ref{disk} to 
handlebodies with two spots.

\begin{theo}\label{diskgraph}
The disk graph of a handlebody $H$ of genus $g\geq 2$ with 2 spots  
on the boundary contains 
quasi-isometrically embedded
$\mathbb{R}^2$. If $g$ is even, then it contains 
quasi-isometrically embedded $\mathbb{R}^3$.
In particular, it is not hyperbolic.
\end{theo}



The proof of Theorem \ref{diskgraph} uses the presence of precisely two spots 
in an essential way, and it is not an extension of the proof of the second 
part of Theorem \ref{disk}.  

Theorem \ref{diskgraph} shows that  
disk graphs can not be used 
effectively to obtain a geometric understanding
of the handlebody group ${\rm Map}(H_0)$ of a handlebody
$H_0$ of genus
$g\geq 3$ with no spots 
paralleling the program developed by Masur and 
Minsky for the mapping class group \cite{MM00}.

The analogue of the strategy of Masur and Minsky 
would consist in taking advantage of 
hyperbolicity of the 
disk graph on which the handlebody group acts 
coarsely transitively. One then 
analyzes the point stabilizers for this action. 
That this is a valuable approach for the geometric study
of the handlebody group follows from the fact that
the stabilizer of a disk is an undistorted subgroup of the 
handlebody group \cite{He21}. 
But the second part of Theorem \ref{disk} and 
Theorem \ref{diskgraph} yield that the geometry of 
the stabilizer of a disk
can not be studied effectively by 
cutting a handlebody open along an embedded disk
which results in a (perhaps disconnected) handlebody
with two spots on the boundary and, in an inductive
procedure, using
knowledge of the  geometry of the disk
graph of the cut open handlebody. 
This failure of 
such an inductive approach may be a witness
for the fact that ${\rm Map}(H_0)$ is an exponentially distorted subgroup of 
the mapping class group of $\partial H_0$ \cite{HH12}, and
its Dehn function is exponential \cite{HH19}. 

Theorem \ref{diskgraph} has a stronger analogue for geometric graphs
related to the outer automorphism group ${\rm Out}(F_g)$ 
of the free group on $g$ generators. 
Namely, doubling the handlebody $H$ yields a connected sum
$M=\sharp_gS^2\times S^1$ of $g$ copies of $S^2\times S^1$
with $m$ marked points.
A doubled disk is an embedded essential sphere in $M$, that is, a 
sphere which is not homotopically trivial or homotopic into a marked point. 
The \emph{sphere graph} of $M$ is the graph whose
vertices are isotopy classes of 
essential spheres in $M$ and where two such
spheres are connected by an edge of length one if and 
only if they can be realized disjointly. 
As before, an isotopy of spheres is required to be
disjoint from the marked points.
The sphere graph of a doubled handlebody without
marked points is hyperbolic \cite{HM13}.
If $g$ is even, then the sphere graph of a doubled handlebody
with one marked point on the boundary contains
quasi-isometrically embedded $\mathbb{R}^2$ \cite{H12}. 


The following is the second main result of this article.

\begin{theo}\label{spheregraph}
The sphere graph of a doubled handlebody
of genus $g\geq 4$ with $2$ marked points
contains for every $n\geq 1$ a quasi-isometrically embedded
copy of $\mathbb{R}^n$. In particular, it is not hyperbolic, and 
its asymptotic dimension is infinite. 
\end{theo}

The sphere graph of a doubled handlebody of genus $g\geq 2$ with
two marked points is isomorphic to the subgraph of the
sphere graph of a doubled handlebody $\sharp_{g+1}S^1\times S^2$ 
of genus $g+1$ with no marked point consisting
of all spheres which are disjoint
from a fixed non-separating sphere (see Section \ref{freesplittings}
for details). As in the case of disks in a handlebody and the
handlebody group, for $g\geq 3$ the stabilizer of 
a sphere in $\sharp_g S^1\times S^2$ is an undistorted subgroup of 
${\rm Out}(F_g)$ \cite{HM10} and thus Theorem \ref{spheregraph}
may among others witness the fact that the Dehn function of
${\rm Out}(F_g)$ is exponential \cite{BV12}. 
Note that unlike for the disk graph,
it seems to be unknown whether or not for $h\geq 3$ 
the sphere graph of $\sharp_h S^1\times S^2$
has finite asymptotic dimension.

The first example known to us of a 
geometric graph
of infinite asymptotic dimension is due to Sabalka and Savchuk
\cite{SS14}.  The vertices of this graph are 
isotopy classes of 
essential separating 
spheres in $\sharp_gS^2\times S^1$.
Two such spheres are connected 
by an edge of length one if and only if they can be
realized disjointly. We use the main construction of \cite{SS14} for the proof 
of Theorem \ref{spheregraph}. 

Note that 
Theorem \ref{diskgraph} and Theorem \ref{spheregraph} do not exclude the possibility
that the graph of non-separating disks or non-separating spheres in a
handlebody with two spots or a doubled handlebody with two spots is hyperbolic.

The article is subdivided into four sections.
In Section \ref{disksenclosing}, we introduce disks in a handlebody of
genus $g\geq 2$ enclosing the two spots.
We show how one can pass from a disk $D$ enclosing the two spots to another
disk $E$ enclosing the two spots in two explicit but different ways with 
a point pushing diffeomorphism. Namely, we can use
point pushing of either of the two spots, resulting in different elements of the handlebody group. 
The effect on disks of such point pushing transformations 
can be controlled provided that they 
are supported in disjoint subsurfaces of the boundary of the handlebody. 
As a byproduct of this analysis, we obtain the following statement of independent interest.

\begin{theo}\label{main3}
Let $\Sigma$ be a compact oriented surface with nonempty boundary $\partial \Sigma$, possibly 
with a finite number of interior points (punctures) removed. 
Let $p_1\in \partial \Sigma$ be a fixed point and let $p_2\in \Sigma$ be a marked interior point. 
Then point pushing the marked point $p_2$ determines a bijection between $\pi_1(\Sigma,p_2)$ and the set 
${\cal A}(p_1,p_2)$ of 
isotopy classes of arcs in $\Sigma$ connecting $p_1$ and $p_2$. 
\end{theo}

The boundary of a handlebody $H$ of even genus $2h$ contains preferred subsurfaces carrying
the entire topology of the handlebody. These subsurfaces are compact surfaces $F$ of genus $h$
with connected boundary so that $H$ is an orientable $I$-bundle over $F$
(that is, a fiber bundle over $F$ with fiber the interval $[0,1]$). The orientation reversing
involution $\Psi$ of $H$ which exchanges the endpoints of the fibers maps $F$ to a disjoint subsurface
of the boundary. If we choose one of the spots $p$ in the interior of $F$ and assume that the second spot 
equals $\Psi(p)$, then point pushing of $p$ along loops in $F$ commutes with point pushing of 
$\Psi(p)$ along loops in $\Psi(F)$. This is explained in detail in 
Section 3. A variation of this construction extends
to handlebodies of odd genus. 
We use this to prove Theorem \ref{diskgraph}.

In Section 4, we apply the discussion in Section \ref{ibundles}
and the main construction
of \cite{SS14} to the double of a handlebody with two spots and 
establish
Theorem \ref{spheregraph}. 

\bigskip
\noindent
{\bf Acknowledgement:} 
I am grateful
to Peter Teichner for pointing out that Theorem \ref{main3} may be of independent interest.

\section{Disks enclosing the spots}\label{disksenclosing}

The goal of 
this section is to introduce disks enclosing the two spots, establish some 
first properties of these disks and prove Theorem \ref{main3}. 
To simplify the terminology, when we talk about disks in the 
sequel, we always identify disks which are isotopic in the sense
explained in the introduction.

We begin with discussing briefly handlebodies of genus one.
A handlebody of genus one with at most one spot on the
boundary contains a single disk up to isotopy. 
This is used to establish

\begin{proposition}\label{twicepuncturedtorus}
The disk graph of a solid torus with two spots 
on the boundary is a tree.
\end{proposition}
\begin{proof} Let $H$ be a solid torus with two 
spots $p_1,p_2$ on the boundary.
The handlebody $H_1$ obtained from $H$ by
removing the spot $p_2$ is a solid torus
with one spot on the boundary. Let
$\Phi_1:\partial H\to \partial H_1$ be the natural
spot removal map.  

The handlebody 
$H_1$ contains a single disk $D_1$, and this
disk is non-separating.  
If $D\subset H$ is any non-separating disk then 
$\Phi_1(\partial D)=\partial D_1$.  
Thus by Theorem 7.1 of  
\cite{KLS09}, 
the complete subgraph 
of the disk graph of $H$ whose vertex set is the
set of non-separating
disks in $H$ is a tree $T$. This is the Bass-Serre tree for the
splitting of $\pi_1(\partial H_1,p_2)=F_2$ defined by $D_1$. Equivalently,
it is the tree dual to the curve $\partial D_1$ and its images under the 
action of $\pi_1(\partial H_1,p_2)$.

If $D\subset H$ is a separating disk then $\partial D$ 
decomposes $\partial H$ into a disk with two spots and a torus
with the interior of a closed disk removed. In particular, 
$\Phi_1(\partial D)$ is peripheral. 
There is a single disk in $H$ which is disjoint from $D$,
and this disk is non-separating. Thus there is a single edge
in ${\cal D\cal G}$ with one endpoint at $D$. The second endpoint
is a vertex in the simplicial tree $T$.
 
As a consequence, the disk graph of $H$ is an extension
of the simplicial tree
$T$ which attaches to each vertex of $T$ 
an at most countable collection of edges 
whose second endpoints are univalent. Thus this graph is a tree as well.
The proposition follows.
\end{proof}

\begin{remark}\label{countable}
The disk graph of a solid torus $H$ with two spots on the boundary is 
a tree with countable valency. Namely, 
if $D$ is any non-separating disk in $H$ then cutting
$H$ open along $D$ yields a ball with four spots on the boundary.
Two of these spots are the two copies of $D$. Any simple closed
curve which separates these two distinguished spots from the 
remaining two spots is the boundary of 
a separating disk in $H$ disjoint from $D$, and any separating disk disjoint from $D$
arises in this way. There are countably many such disks. 
\end{remark}


\begin{remark}\label{connected}
It also follows from similar considerations that the disk graph of a handlebody
$H$ with two spots is connected (for which the usual surgery argument 
\cite{H19a} is problematic as surgery may lead to peripheral disks). 
Namely, denote by 
$H_1$ the handlebody obtained from $H$ by removing the spot $p_2$, 
and let $\Phi_1:H\to H_1$ be the spot forgetful map. 
By Theorem 7.1 of \cite{KLS09}, for any
non-separating disk $D\subset H$, 
the preimage of $\Phi_1(D)\subset H_1$ under the map 
$\Phi_1$ is the Bass Serre tree for the graph of groups decomposition 
of $\pi_1(\partial H_1,p_2)$ defined by $D$ and the point $p_2\in \partial H_1-\Phi_1(D)$. 
In particular, this preimage is connected. On the other hand, any disk in $H$
is disjoint from a disk which projects onto a non-separating disk in $H_1$. 
As the disk graph ${\cal D\cal G}_1$ 
of $H_1$ is easily seen to be connected and furthermore there
is a simplicial embedding $\Lambda: {\cal D\cal G}_1\to {\cal D\cal G}$ so that 
$\Phi_1\circ \Lambda={\rm Id}$ (see Section 7 of \cite{KLS09}), 
this yields that ${\cal D\cal G}$ is indeed connected.  
\end{remark}

In the sequel we always denote by $H$ a handlebody of genus $g\geq 2$ 
with two spots $p_1,p_2$ on the
boundary, and we denote by $H_0$ the handlebody obtained from $H$ by
removing the spots. There is a natural spot removal map 
\[\Phi:H\to H_0.\] 
We say that a disk $D$ \emph{encloses the spots} if $\Phi(D)\subset H_0$ is contractible.
Equivalently, the boundary of $D$ is a simple closed curve
in $\partial H$ which bounds a twice punctured disk $\tilde D\subset \partial H$, with punctures at
$p_1,p_2$.

Let $E\subset H$ be another disk which encloses the
two spots $p_1,p_2$. Assume that $E$ is 
in minimal position
with respect to $D$. This means in particular
that the boundaries $\partial D,\partial E$ 
intersect in the minimal number of points among all
representatives in their isotopy classes. 

The simple closed curves $\partial D,\partial E$  
are the boundaries of 
unique disks $\tilde D,\tilde E\subset \partial H$ 
containing the two spots (thus if we think of the spots as
missing points, then $\tilde D,\tilde E$ should be viewed as 
twice punctured disks). If $D$ is not isotopic to $E$ then 
the intersection $\tilde D\cap \tilde E$
consists of two disjoint disks $A_1,A_2$ 
with one spot at $p_1,p_2$, respectively,
and a disjoint union of rectangles.
In particular, $\partial D\cap \partial E$ consists
of at least four points.


Denote as before by $H_1$ the handlebody obtained from $H$ by removing the 
spot $p_2$ and let $\Phi_1:H\to H_1$ be the spot forgetful map. 
Use the ordered pair of disks $(\tilde D,\tilde E)$ to 
construct a loop $\gamma\subset 
(\partial H \cup\{p_2\},p_2)$ based at $p_2$ as follows. 
First, connect the point $p_1$ to the point $p_2$
by an oriented arc $\alpha$, that is, the image of the closed 
interval $[0,1]$ under a topological embedding,  
whose interior is embedded in $\tilde D$. 
The endpoints of $\alpha$
are the two spots of $H$, and they are precisely the spots of the
disk $\tilde E$. Thus there is an arc $\beta$ in 
$\tilde E$ connecting $p_1$ to $p_2$. The loop $\gamma$ is 
homotopic to the concatenation of $\alpha^{-1}$ with $\beta$
(which we move off the spot $p_1$ with a small deformation).
Note that
the inverse of the loop $\gamma$ is constructed with exactly the 
same procedure, but with the roles of the disks $D,E$ exchanged. 
Furthermore, the homotopy class of $\gamma$ 
as a loop in $\partial H_1$ based at $p_2$ 
is uniquely determined by
the ordered pair $(D,E)$ up to the  precomposition with the
homotopy class of a loop in $\tilde D$ based at $p_2$ which 
surrounds the marked point $p_1$.   

The fundamental group $\pi_1(\partial H_1,p_2)$ 
of $\partial H_1$ is the free group in 
$2g$ generators. Let $c\in \pi_1(\partial H_1,p_2)$ be the 
element which can be represented by a loop in $\tilde D$ surrounding
the marked point $p_1$. If we write composition from left to right, then 
the above discussion shows that 
each ordered pair $(D,E)$ of disks in $H$ 
enclosing the two spots determines uniquely the right coset 
of a homotopy class in $\pi_1(\partial H_1,p_2)$ 
by the infinite
cyclic group generated by $c$. 

To avoid working with cosets, we now replace the spot $p_1$ in 
$\partial H_1$ by a boundary component. Let $\Sigma$ be the resulting 
bordered surface. Attaching to the  boundary $\partial \Sigma$ 
of $\Sigma$
a disk yields the boundary of the handlebody $H_0$ without spot. 
Fix a point $p_1$ in $\partial \Sigma$ and denote by
${\cal A}(p_1,p_2)$ the set of isotopy classes of arcs in 
$\Sigma$ with fixed endpoints $p_1,p_2$. 
Such an arc can be viewed as an arc in $\partial H$ with endpoints 
at the spots, and a thickening of such an arc 
defines a disk $\tilde E\subset \partial H$
enclosing the two spots. Vice versa, any disk in $H$ enclosing the two 
spots determines an arc in ${\cal A}(p_1,p_2)$, unique up to the ambiguity
of Dehn twisting the spot $p_1$ about the boundary of $\Sigma$. 
Thus understanding isotopy classes of disks enclosing the two spots 
amounts to understanding ${\cal A}(p_1,p_2)$.

More generally, let for the moment $\Sigma$ be any compact surface with
non-empty boundary $\partial \Sigma$ and perhaps a finite number of 
points removed. Choose a point $p_2$ in the interior of $\Sigma$.
This choice determines a subgroup of the 
mapping class group ${\rm Mod}(\Sigma-\{p_2\})$ of $\Sigma-\{p_2\}$ which
is isomorphic to the fundamental group $\pi_1(\Sigma,p_2)$ 
of the surface $\Sigma$. 
This group is the fiber group of 
the \emph{Birman exact sequence} 
\[0\to \pi_1(\Sigma,p_2)\to {\rm Mod}(\Sigma-\{p_2\})\to {\rm Mod}(\Sigma)\to 0\]
obtained from the map $\Sigma-\{p_2\}\to \Sigma$ 
which forgets the spot $p_2$. Its elements are called
\emph{point pushing maps}. 
If $\gamma\subset \Sigma$ is a based loop at $p_2$, then the point pushing map
along $\gamma$ can be represented by a diffeomorphism supported in 
an arbitrarily small neighborhood of $\gamma$. 

Let as before ${\cal A}(p_1,p_2)$ be the set of isotopy classes of arcs with endpoints 
$p_1,p_2$. For ease of exposition, we view such arcs $\alpha,\beta$ as arcs 
connecting
$p_1$ to $p_2$. Then the composition $\alpha^{-1}\circ \beta$ (read from left to right) is a based 
loop at $p_2$ whose homotopy class we denote by 
$q(\alpha,\beta)$. Note that $q(\alpha,\beta)\in \pi_1(\Sigma,p_2)$ 
only depends on the isotopy class of 
$\alpha,\beta$. For easier exposition, in the sequel we always represent isotopy classes
of arcs by actual arcs and note that the statements we make do not depend on the choice of 
such representatives.

\begin{lemma}\label{pointpush1}
Let $\alpha,\beta\in {\cal A}(p_1,p_2)$. Then $\beta$ is the image of 
$\alpha$ under the element of the mapping class group of $\Sigma-\{p_2\}$ obtained by
pushing the point $p_2$ along a loop in the homotopy class 
$q(\alpha,\beta)$. 
\end{lemma}
\begin{proof} Let $\gamma\subset \Sigma\cup \{p_2\}$ 
be a loop based at $p_2$ constructed
as above from the ordered pair $(\alpha,\beta)$ of arcs, 
and moved off $\Sigma  \ni   p_1$.

Assume first that 
the loop $\gamma$ is simple,
that is, it does not have self-intersections. 
Thus $\gamma$ is embedded
in $\Sigma$ and hence
there is an embedded annulus $A\subset \Sigma$ with core 
curve $\gamma$, disjoint from $p_1$.  
Up to isotopy, the arc $\alpha$ intersects 
$A$ in a single embedded segment $\alpha_0$ with one endpoint $p_2$. Furthermore, 
we may assume that $\alpha_0$ intersects the core curve $\gamma$ 
of $A$ in the unique point $p_2$.

The point pushing homeomorphism of $p_2$ along $\gamma$ equals the identity on 
$(\Sigma-A)\cup \gamma$. In each of the two components of $A-\gamma$, it is a 
Dehn twist about the core curve (which is 
freely homotopic to $\gamma$), in one component positive, in the second negative. 
As $\alpha_0$ and hence $\alpha$ meets only one component of $A-\gamma$, this homeomorphism
transforms the homotopy class of $\alpha$ with fixed endpoint by
concatenation with $\gamma$. As a consequence, the image of the arc $\alpha$ by the 
point pushing map along $\gamma$ is homotopic to the concatenation of 
$\alpha$ with $\gamma$ 
(read from left to right). 
This arc is homotopic with fixed endpoints to $\beta$, and in fact 
isotopic to $\beta$.

This construction extends to the case that the curve $\gamma$ has self-intersections.
Namely, parameterize the arc $\beta$ on the interval $[0,1]$.
Assuming that $\beta$ is in minimal position with respect to $\alpha$, let 
$0=t_0<t_1<\cdots <t_k=1$ be such that $\beta(t_i)$ are the intersection
points of $\beta$ with $\alpha$, with the endpoints included. 
Let $\gamma_1$ be the concatenation of $\alpha^{-1}$ with an arc $\zeta_1$ connecting 
$p_1$ to $p_2$ which is composed of $\beta[0,t_1]$ and the subarc $\alpha_1$ of 
$\alpha$ connecting $\beta(t_1)$ back to $p_2$.
With a small homotopy with fixed endpoints,
the arc $\zeta_1$ can be pushed off the interior
of $\alpha$. The resulting loop based at $p_2$ is simple and can
be pushed off $p_1$. Denote this loop again by $\gamma_1$.

In a second step, define a based loop $\gamma_2$ at $p_2$ 
as follows. Let $2\leq j\leq k$ be such that $\beta(t_j)$ is the first intersection point of
$\beta(t_1,1]$ with the subarc $\alpha_1$ of $\alpha$ 
(perhaps this is the endpoint
$p_2$ of $\beta$). Let $\gamma_2$ be the concatentation of 
$\zeta_1^{-1}$, the subarc $\beta[0,t_j]$ of $\beta$ and the subarc 
$\alpha_2$ of $\alpha$ connecting
$\beta(t_j)$ back to $p_2$. This loop is homotopic with fixed endpoints to the
concatenation of $\alpha_1^{-1}$, the arc $\beta[t_1,t_j]$ and the arc $\alpha_2$ and hence
it is simple. Furthermore, it can also be described as the concatentation
of $\zeta_1^{-1}$ and an arc $\zeta_2$ connecting $p_1$ to $p_2$.  

Proceeding inductively, define for $\ell\leq m$ (where $m\leq k$ is a number
computed from the order in which the points of $\alpha\cap \beta$ 
are passed through by $\beta$) a based 
loop $\gamma_\ell$ at $p_2$ composed of 
the arc $\zeta_{\ell-1}^{-1}$ and an arc $\zeta_\ell$ connecting 
$p_1$ to $p_2$ which is a concatenation of a  subarc $\beta[0,t_{j(\ell)}]$ of $\beta$ and the subarc of $\alpha$ connecting 
$\beta(t_{j(\ell)})$ back to $p_2$.
Up to homotopy, 
the loops $\gamma_\ell$ are simple, and we have
$\gamma=\gamma_1\circ \cdots \circ \gamma_m$ (read from left to right). 

By the first part of this proof, the image of $\alpha$ under
point pushing along $\gamma_1$
is the arc  $\zeta_1$.
By the above construction, the arc $\zeta_1$ intersects $\gamma_2$ only at the point $p_2$. 
Thus using
again the first part of this proof, point pushing of $\zeta_1$ 
along the loop $\gamma_2$ yields the 
arc $\zeta_2$. As the point pushing group is a group, we also know that $\zeta_2$ is the
image of $\alpha$ under point pushing along $\gamma_1\circ \gamma_2$ (read from left to right). 
In $m\leq k$ such steps we deduce that indeed, the arc $\beta$ is isotopic to the image of 
$\alpha$ under pointpushing of $p_2$ along $\gamma$.
\end{proof}

As a corollary, we obtain Theorem \ref{main3} from the introduction.

\begin{corollary}\label{arcs}
Let $\alpha\in {\cal A}(p_1,p_2)$ be a fixed basepoint. Then the 
map which associates to an element $\gamma\in \pi_1(\Sigma,p_2)$ 
the image of $\alpha$ under point pushing along $\gamma$ defines a bijection 
$\pi_1(\Sigma,p_2)\to {\cal A}(p_1,p_2)$. 
\end{corollary}
\begin{proof}
Lemma \ref{pointpush1} shows that up to homotopy, 
any arc $\beta\in {\cal A}(p_1,p_2)$ can 
be obtained from the fixed arc $\alpha$ by point pushing along a loop 
$\gamma$. 
Since two point pushing homeomorphisms along homotopic loops are isotopic, the 
isotopy class of the resulting arc only depends on the homotopy class of the 
loop $\gamma$. Thus the map $\Lambda$ which associates to a homotopy class 
$[\gamma]\in \pi_1(\Sigma,p_2)$ the isotopy class of the 
image of $\alpha$ by point pushing 
along a based loop in the class $[\gamma]$ is surjective.

On the other hand, using again Lemma \ref{pointpush1} and its proof,
if $\beta$ is obtained from $\alpha$ by point pushing along a based loop 
$\gamma$, then the homotopy class of $\beta$ with fixed endpoints equals the 
homotopy class of the concatenation $\alpha \circ \gamma$.
This means that  
point pushing along non-homotopic loops
gives rise to arcs in different homotopy classes with fixed endpoints and hence
to arcs which are not isotopic. This shows that the map $\Lambda$ is injective
as well.
\end{proof}

Corollary \ref{arcs} needs to be slightly modified to obtain a statement about disks
in the handlebody $H$ enclosing the two spots, or,
equivalently, arcs with fixed endpoints on a compact oriented surface, where the
endpoints are interior points of the surface. Namely, looking again at two 
disks $\tilde D,\tilde E$ enclosing the two spots $p_1,p_2$ and arcs 
$\alpha\subset \tilde D,\beta\subset \tilde E$  connecting the two spots, 
point pushing $\tilde D$ along $\gamma= \alpha^{-1}\circ \beta$ requires  moving 
$\gamma$ off $p_1$ which depends on a choice. Two choices differ by 
the homotopy class of a based loop at $p_2$ which is entirely contained in 
$\tilde D$ and encloses the spot $p_1$.  
The next lemma shows that
the disk $E$ is independent of the choice made.

\begin{lemma}\label{nochoice}
The isotopy class of the disk $\tilde D$ is fixed by point pushing $p_2$  
along a loop $\zeta$ based at $p_2$ which is entirely contained in the disk $\tilde D$
and encircles $p_1$.
\end{lemma}
\begin{proof} 
Up to homotopy with fixed basepoint, 
the loop $\zeta$ is embedded in a disk $\hat D\subset \tilde D$ 
which is isotopic to $\tilde D$, in particular, it 
contains the two marked points $p_1,p_2$ in its interior, and whose closure in 
$\partial H$ is contained in the interior of $\tilde D$. 
Point pushing of $p_2$ along $\zeta$ can be represented by a diffeomorphism which 
fixes $\partial H-\hat D$ and hence $\partial \tilde D$
pointwise. As a consequence,
point pushing
of $p_2$ along $\zeta$ preserves the disk $\hat D$.
\end{proof}

\begin{remark}
Replacing one of the spots by a boundary component and marking a point on this
component removes the ambiguity in the point pushing construction. 
\end{remark}

Let us summarize what we obtained so far.
Let $D\subset H$ be a disk enclosing the two spots
$p_1,p_2$.  
Its boundary $\partial D$ is a simple closed curve in $\partial H$. 
It determines the homotopy class of an arc in 
$\partial H_0$ with endpoints 
$p_1,p_2$. 
Let us choose such an arc $\alpha\subset \partial H_0$, oriented 
in such a way that it connects
$p_1$ to $p_2$, and let $a$ be its homotopy class with fixed endpoints
as an arc in $\partial H_0$. 

Let $E\subset H$ be another such disk which determines the homotopy class
$b$ of an oriented arc $\beta\subset \partial H_0$
connecting $p_1,p_2$. The concatenation 
$\alpha^{-1}\circ \beta$ (read from left to right) 
is a loop based at $p_2$. It defines a right coset 
$a^{-1}\cdot b\in \langle c \rangle\backslash \pi_1(\partial H_1,p_2)$
where $c$ is a loop encircling the spot $p_1$ which is entirely contained in the 
twice punctured disk $\tilde D\subset \partial H$ whose boundary 
equals the boundary of $D$.

Move as before $\alpha^{-1}\circ \beta$ off $p_1$.
Point pushing $p_2$ along the this defined loop 
$\alpha^{-1}\circ \beta$ in $\partial H_1$ 
defines the isotopy class of a 
homeomorphism of $\partial H$.
It can be represented by a homeomorphism
which equals the identity on the complement of a small neighborhood of 
$\alpha^{-1}\circ \beta$, and we may assume that the point 
$p_1$ is not contained in this neighborhood.
The image of the disk $D$ under this homeomorphism
is the disk $E$. Thus Lemma \ref{pointpush1} describes an algorithm which 
begins with a pair of disks $(D,E)$ enclosing the two spots and 
determines from this pair up to an ambiguity arising from choosing how to 
move off the spot $p_1$ 
a point pushing homeomorphism of $\partial H$ which transforms $D$ into $E$. 
As this ambiguity corresponds to the choice of a representative of the corresponding 
coset $\langle c\rangle \backslash \pi_1(\partial H\cup \{p_2\},p_2)$, we conclude 
the following

\begin{corollary}\label{diskenclosing}
The point pushing map induces a bijection from the set of cosets 
$\langle c \rangle \backslash \pi_1(\partial H\cup \{p_2\},p_2)$ onto the set of isotopy classes of 
disks in $H$ enclosing the two spots. 
\end{corollary} 

The same construction is also valid for point pushing the point $p_1$. In this
case, the point pushing map constructed in the above
fashion which transforms $D$ into $E$ is point pushing of $p_1$ along the based loop 
$\alpha\circ \beta^{-1}$, moved off $p_2$.

\section{$I$-bundles and disk graphs}\label{ibundles}

While Lemma \ref{pointpush1} and Corollary \ref{arcs} 
do not use specific properties of disks and can be viewed as 
statements about essential simple closed curves on $\partial H$
which become homotopically trivial after closing the spots, 
we now turn to the study of disks enclosing the two spots as vertices
of the disk graph of $H$.
Note that the description of a
disk $E$ enclosing the two spots as the image of 
a disk $D$ by point pushing along an element in 
$\pi_1(\partial H\cup \{p_2\},p_2)$  
established in Lemma \ref{pointpush1} coarsely 
determines a path in the curve graph of $\partial H$  
connecting
$\partial D$ to $\partial E$ (that is, consecutive vertices intersect in uniformly
few points and hence are uniformly close in the curve graph), 
but such a path may be 
highly inefficient. 
In fact, as the curve graph of $\partial H$ is hyperbolic and such paths are also paths in the disk
graph of $H$, it can be deduced from the proof 
of Theorem \ref{diskgraph} below that this
is indeed the case.

Following \cite{H19a,H16}, 
define an \emph{$I$-bundle generator} for $H_0$ to be
a simple closed curve $c\subset \partial H_0$ so that
$H_0$ can be realized as an $I$-bundle over a compact
surface $F$ with connected boundary
$\partial F$ and such that $c$ is the core curve of the 
vertical boundary of the $I$-bundle.\footnote{The assumption of 
connectedness of the boundary of the base of the $I$-bundle is not
used in \cite{H19a,H16} but it is needed here to ensure that
the $I$-bundle generator is diskbusting.}  
This vertical boundary is an annulus
bounded by the two preimages of $\partial F$.
The surface $F$ is called the \emph{base} of the $I$-bundle.
If the $I$-bundle generator $c$ is separating, then 
$F$ is orientable of genus $g/2$ where $g$ is the genus of $H_0$.
If $c$ is non-separating, then the surface $F$ is non-orientable, and
the complement of an open annulus about $c$ in 
$\partial H_0$ is the orientation cover of $F$. 
The $I$-bundle over every essential 
arc in $F$ with endpoints on $\partial F$ is an essential
disk in $H_0$ which intersects $c$ in precisely two points (up to isotopy).

An $I$-bundle generator
$c$ in $\partial H_0$ is \emph{diskbusting}, 
which means that it has an essential 
intersection with every disk (see \cite{MS10,H19a}). 
Namely, the base $F$ of the $I$-bundle is a deformation retract of $H_0$. 
Thus if $\gamma$ is any essential closed curve on $\partial H_0$ 
which does not intersect $c$ then $\gamma$ projects to 
an essential closed curve on $F$. 
Such a curve is not nullhomotopic in $H_0$ and hence it can not be diskbounding.

The \emph{arc graph} ${\cal A}(X)$ of a compact 
surface $X$ of genus $n\geq 1$ with 
connected boundary $\partial X$ and possibly marked points (punctures)
in the interior of $X$ is the graph whose
vertices are isotopy classes of 
essential arcs in $X$ with endpoints on the
boundary, and isotopies are allowed to move the endpoints 
of an arc along $\partial X$. 
Two such arcs are connected by
an edge of length one if and only if they can
be realized disjointly. The arc graph ${\cal A}(X)$ of $X$ is 
hyperbolic, however the inclusion of ${\cal A}(X)$ into 
the arc and curve graph of $X$ is a 
quasi-isometry only if $X$ is of genus one, with at most one marked point
\cite{MS10} (see also
\cite{H16}).

For an $I$-bundle generator $c$ in $\partial H_0$ let 
${\cal R\cal D}(c)$ be the complete subgraph of the disk graph 
${\cal D\cal G}_0$ of 
$H_0$ consisting of disks which intersect $c$ in precisely two points.
Each such disk is an $I$-bundle over 
an arc in the base $F$ of 
the $I$-bundle corresponding to $c$. We refer to the discussion preceding
Lemma 2.1 of \cite{H12} for details. 
As two such disks are disjoint
if and only if the corresponding arcs in $F$ are disjoint, the
graph ${\cal R\cal D}(c)$ is isometric to the arc graph 
${\cal A}(F)$ of $F$. 
The following is a consequence
of \cite{MS10,H16,H19a}. We refer to Lemma 2.2 of \cite{H12} for a
detailed proof.

\begin{lemma}\label{casewithoutspot}
For each $I$-bundle generator $c$ of $\partial H_0$, 
the inclusion ${\cal R\cal D}(c)\to {\cal D\cal G}_0$ is a quasi-isometric
embedding.
\end{lemma}

As the $I$-bundle over an arc with endpoints on $\partial F$ intersects the curve $c$ in 
precisely two points, such $I$-bundles over arcs with large distance in the arc
graph of $F$ give rise to disks with  
large distance in the disk graph of $H_0$. 
However, the distance in the curve graph of their boundary
circles is at most $4$. In \cite{MS10}, such a subspace of the disk graph is 
called a hole. 


Let as before $F$ be the base of the $I$-bundle determined by
the $I$-bundle generator $c$.
Let $\Psi$ be the involution of $H_0$ which exchanges the two endpoints of the intervals
in the interval bundle. Its fixed point set intersects $\partial H_0$ in the $I$-bundle
generator $c$. Let $\tilde F\subset \partial H_0$ be the preimage of $F$; this is 
the complement of an annulus in $\partial H_0$ with core curve $c$.
Choose a point $p_1\in \partial \tilde F$, let 
$p_2=\Psi(p_1)$ and define $H=H_0-\{p_1,p_2\}$. 
The boundary of the thickening of the 
interval with endpoints $p_1,p_2$ is the boundary of a disk $D$  
enclosing the two spots. 
A disk in ${\cal R\cal D}(c)$ which is disjoint from $p_1,p_2$ defines
a disk in $H$ which is invariant under $\Psi$ up to isotopy and is disjoint from $D$.

Push the points $p_1,p_2=\Psi(p_1)$ slightly into the interior of $\tilde F$
so that $\tilde F$ can be thought of as a two-sheeted cover of a surface   
$F^+$ with connected boundary and one marked point (spot) in its interior.
Let ${\cal A}(F^+)$ be the arc graph of the surface $F^+$, and let 
${\cal R\cal D}^+(c)$ be the subgraph of the disk graph
${\cal D\cal G}$ of $H$ whose vertices 
are $\Psi$-invariant disks, with boundaries  
intersecting the fixed point set of $\Psi$ in precisely two points.

For some $\ell\geq 1$, a  \emph{coarse $\ell$-Lipschitz retraction}
of a geodesic metric space $X$ onto a subspace $Y\subset X$ is a coarse 
$\ell$-Lipschitz map $E:X\to Y$ (that is, a map which is $\ell$-Lipschitz up to
a uniform additive constant) 
such that there exists a number $k>0$ with 
$d(Ey,y)\leq k$ for all $y\in Y$. The subspace $Y\subset X$ is quasi-isometrically
embedded in $X$. Since $X$ is a length space and quasi-geodesics
do not have to be continuous, this is equivalent to stating
that any two points in $Y$ can be connected by a uniform quasi-geodesic in $X$ 
which is entirely contained in $Y$.
In analogy to  
Lemma \ref{casewithoutspot}, we have

\begin{lemma}\label{casewithspot}
There exists a coarse two-Lipschitz retraction 
$\Omega:{\cal D\cal G}\to {\cal R\cal D}^+(c)$. 
Furthermore, ${\cal R\cal D}^+(c)$ is isometric 
to the arc graph ${\cal A}(F^+)$.
\end{lemma}
\begin{proof} A disk which is $\Psi$-invariant and 
whose boundary intersects the fixed point set of 
$\Psi$ in precisely two points is the $I$-bundle over an arc in
$F^+$, and two such disks are disjoint if and only if their defining arcs in
$F^+$ are disjoint. Furthermore, the $I$-bundle over  
any nontrivial arc $\alpha$ in $F^+$ is a disk in $H$. 
Thus 
the map which associates to an arc $\alpha \in {\cal A}(F^+)$ 
the $I$-bundle over $\alpha$ is an isometry of
${\cal A}(F^+)$ onto ${\cal R\cal D}^+(c)$.

Let $d_{\cal D\cal G}$ be the distance in the disk
graph ${\cal D\cal G}$ 
of $H$. To construct a coarse Lipschitz
retraction $\Omega:{\cal D\cal G}\to {\cal R\cal D}^+(c)$ 
we proceed along the lines of the proof of Lemma 2.2 of \cite{H12}. Assume first 
that the $I$-bundle generator $c$ is separating. Then $F^+$ can be identified with
a subsurface of $\partial H$. There are two different choices for such an
identification, and we pick one of them.

Let $D\subset H$ be any disk. Since the two marked points on $\partial H$ 
are contained in different
components of $\partial H-c$ and the image of $c$ under the spot forgetful map 
$\Phi:H\to H_0$ is diskbusting, 
the boundary of $D$ intersects $c$ and hence $F^+$. Take any 
component of $F^+\cap \partial D$ and associate to $D$ the $\Psi$-invariant disk 
$\Omega(D)$ which is the $I$-bundle over 
this intersection component. The disk $\Omega(D)$  coarsely does not depend on
choices: Another choice of intersection arc gives rise to a disjoint disk.
Moreover, the images under $\Omega$ of disjoint disks are disjoint and hence  
this construction defines a coarse one-Lipschitz map
$\Omega:{\cal D\cal G}\to {\cal R\cal D}^+(c)$ which satisfies $\Omega(D)=D$ for every
$D\in {\cal R\cal D}^+(c)$. Thus $\Omega$ is a coarse one-Lipschitz retraction, and  
the inclusion ${\cal R\cal D}^+(c)\to {\cal D\cal G}$ is indeed a quasi-isometric
embedding. This completes the proof for separating 
$I$-bundle generators.

This construction can be modified to cover the case of a non-separating
$I$-bundle
generator $c$ as well, that is,  when $F^+$ is a non-orientable surface with 
one marked point. Namely, 
a non-orientable surface $F^+$ of Euler characteristic $-2h$ with connected boundary $\partial F$
and one marked point 
can be represented as the connected sum of an orientable surface of genus $h$ with connected 
boundary and one marked point in the interior 
and a projective plane. Equivalently, $F^+$ contains an orientable subsurface 
$F_0^+\subset F^+$ with two boundary components
$\partial F^+,e$, and $F^+$ is obtained from $F_0^+$ by gluing a  
M\"obius band to the boundary component $e$.  
The fundamental group of $F_0^+$ is an index two subgroup of the fundamental group of $F^+$. 
As the surface $F_0^+$ is oriented, its preimage in the orientation 
cover $\tilde F^+$ of $F^+$ consists of two disjoint copies of $F_0^+$, and $\tilde F^+$ is
obtained from these two copies of $F_0^+$ by connecting the two components  $e_1,e_2$ 
of the preimage of $e$ with an annulus (which is the orientation cover of the M\"obius band). 
The oriented 
$I$-bundle over $F^+$ contains the trivial $I$-bundle over the bordered subsurface $F_0^+$ as a submanifold.

Identify $F_0^+$ with a subsurface of $\partial H$ containing the marked point $p_1$.
The boundary of $F_0^+$ consists of a simple closed curve $c_1$ isotopic to $c$
and a component $e_1$ of the preimage of the boundary of the M\"obius band. 
Denote as before by $\Psi$ the orientation reversing involution of $H$ defined by
the $I$-bundle. 
Let $D\in {\cal D\cal G}$ be any disk. Since an $I$-bundle generator
in $H_0$ is diskbusting and the surface $F_0^+$ contains
one but not both of the marked points, 
the boundary curve $\partial D$ of $D$ 
intersects the surface $F_0^+$ nontrivially in a collection of 
pairwise disjoint arcs.  
%
%
%
If there is such an arc $\alpha$ with both endpoints on $c_1$ then 
the $I$-bundle over $\alpha$ is a disk. Define $\Omega(D)$ to be this disk.
Its intersection with $F_0^+$ is disjoint from $\partial D$. 
The disk $\Omega(D)$ is coarsely well defined, that is, choosing another
component of $\partial D\cap F_0^+$ with both endpoints on
$c_1$ gives rise to a disjoint disk.

Assume next that $\partial D$ does not contain
an arc with both endpoints on $c_1$ but that there is such an intersection
arc $\alpha$ with one endpoint on $c_1$ and the second endpoint $y$ on $e_1$.
Let $\hat \alpha$ be the projection of $\alpha$ to $F^+$. 
Connect the projection of $y\in e_1$ to the core curve of the
M\"obius band, attach an arc making one full turn around the core
curve of the M\"obius band (in either direction), connect back to
the projection of $y$ and backtrack to $\partial F^+$ with 
the inverse of $\hat \alpha$. Up to homotopy, this construction
defines an embedded 
arc in $F^+$  with both endpoints on $\partial F^+$.
Define $\Omega(D)\in {\cal R\cal D}(c)^+$ to be the $I$-bundle over this arc. 
The disk $\Omega(D)$ depends on the choice of the component 
$\alpha$ of $\partial D\cap F_0^+$ and on the choice 
of the direction of the loop around the core curve of the M\"obius band.
However, any two distinct choices give rise to disks whose boundaries intersect 
in at most two points contained in the preimage of the M\"obius band. 
Such disks are disjoint from the $I$-bundle over some arc in 
$F_0^+$ with both endpoints on $c_1$ and hence their distance
in ${\cal D\cal G}$ is at most two. 

Finally if every component of $\partial D\cap F_0^+$ 
is an arc with both endpoints on $e_1$ then up to homotopy, 
$\partial D$ is disjoint from the simple closed curve 
$c$ in $\partial H$ whose projection to $\partial H_0$ is diskbusting. 
As the projection of $\partial D$ to $\partial H_0$ is diskbounding, this implies that 
the projection of 
$\partial D$ to $H_0$ is contractible, in other words, $D$ encloses the two spots.
In this case let $\alpha$ be an arc in $F_0^+$ disjoint from $\partial D$ which 
connects the spot $p_1\in F_0^+$ to a point $y\in e_1$.
Let $\hat \alpha$
be the projection of $\alpha$ into $F^+$. It connects the spot in $F^+$ to the
boundary $e$ of 
the M\"obius band. As in the previous paragraph,  
connect the projection of $y\in e_1$ to the core curve of the 
M\"obius band, attach an arc making one full turn
around the core curve of the M\"obius band 
 (in either direction)
 and connect back to the spot 
 with the inverse of $\hat \alpha$. Up to homotopy, this defines
 an embedded essential
loop in $F^+$ based at the spot which lifts to a $\Psi$-invariant arc
in $\tilde F^+$ connecting the two spots. Define $\Omega(D)$ to be the 
disk enclosing the two spots whose boundary is the boundary of the thickening 
of this arc. Note that for a fixed choice
of $\hat \alpha$, this construction only depends on the choice of the 
direction of the arc going around the M\"obius band, that is, any two distinct choices
intersect in precisely four points near the spots and are 
uniformly close in the disk graph.

Let $\Omega:{\cal D\cal G}\to {\cal R\cal D}^+(c)$ be the coarsely defined
map constructed above. 
We claim that $\Omega$ is coarsely two-Lipschitz.

To show that this is the case it suffices to show that if $D,D^\prime$ are disjoint, 
then their images intersect in at most two points. That this holds indeed true
for disks whose intersections with $F_0^+$ contain an arc with at least
one endpoint on $c_1$ is immediate from the above discussion.
If this is not the case for say the disk $D$, then $D$ encloses the two
spots. Then  $D^\prime$ does not enclose the two spots 
as two disks enclosing the
two spots are not disjoint. But then for suitable choices made in 
the above construction, the boundaries of the disks
$\Omega(D),\Omega(D^\prime)$  
intersect in at most four points, and these points
are contained in the preimage of the M\"obius band. Once again, the 
distance between  
$\Omega(D),\Omega(D^\prime)$ is at most two.

To complete the proof of the lemma it now suffices to show that 
if $D\in {\cal R\cal D}^+(c)$ then
the distance between $D$ and $\Omega(D)$ is uniformly bounded.
To this end note that if $\partial D\cap F_0^+$ contains an arc with
both endpoints on $c_1$ then $\Omega(D)=D$.

If $\partial D\cap  F_0^+$ contains an arc with one endpoint on $c_1$ and the second 
endpoint on $e_1$, then 
the two arcs $\alpha,\beta\subset F^+$ which define $\partial D,\partial \Omega(D)$
may not be disjoint. However, by construction, up to homotopy these arcs
only intersect in the interior of the M\"obius band. 
As we can find an arc in the surface $F^+$ 
which is disjoint from both $\alpha, \beta$ as well as the M\"obius band, we conclude that
the distance in ${\cal A}(F^+)$ between $\alpha$ and $\beta$ is at most two.
As a consequence, the distance between
$D$ and $\Omega(D)$ in ${\cal R\cal D}^+(c)$ is at most two.

We are left with looking at disks in ${\cal R\cal D}^+(c)$ enclosing the two spots. 
Now $\Omega$ is a coarsely two-Lipschitz map, and every disk in ${\cal R\cal D}^+(c)$ which
encloses the two spots is disjoint from a disk which is the $I$-bundle over
an arc in the surface $F_0^+$, and such a disk is mapped to itself by 
$\Omega$.
Thus by the triangle inequality, we obtain that indeed
$d_{\cal D\cal G}(D,\Omega(D))\leq 3$ for all $D\in {\cal R\cal D}^+(c)$.
The lemma follows.
\end{proof}

\begin{remark}\label{notunique}
The construction in the proof of Lemma \ref{casewithspot} yields in fact two distinct
coarse Lipschitz retractions. If the $I$-bundle generator $c$ is separating, 
then we obtain one such retraction for each of the two components of 
$\partial H-c$. If the $I$-bundle generator is non-separating, then there is one
retraction for each choice of a component of the preimage of the boundary of an 
embedded M\"obius band in the base of the $I$-bundle.
\end{remark}

Let as before $c\subset \partial H$ be an $I$-bundle generator, with 
base $F^+$ and involution 
$\Psi$, and let  
$F_0^+\subset \partial H$ be a once punctured subsurface 
whose boundary either is isotopic to $c$ if $c$ is
separating, or consists of two connected components, say $c_1,e_1$, where
$c_1$ is isotopic to $c$ and $e_1$ is a preimage of the boundary of a M\"obius band
in the base of the $I$-bundle otherwise.

Denote by
${\cal A}(F_0^+)$ the subgraph of the arc graph of $F_0^+$ consisting of arcs
with both endpoints on the boundary component $c_1$ of $\partial F_0^+$. 
If $c$ is separating then ${\cal A}(F_0^+)$ 
equals the arc graph of $F_0^+$,  and it is isometric to the arc graph of the surface
$F^+$.
In the case that $c$ is non-separating, then as the projection of 
$F_0^+$ into $F^+$ is an embedding onto a subsurface of $F^+$ containing the 
boundary, this subsurface is a hole for the arc graph ${\cal A}(F^+)$ of $F^+$ 
in the sense of \cite{MS10}. As a consequence, the 
arc graph ${\cal A}(F_0^+)$ quasi-isometrically embeds into the arc graph ${\cal A}(F^+)$. 

The following statement is in some sense an inverse of Lemma \ref{casewithspot}.
In its formulation, we write $\approx$ to denote an equality up to a universal multiplicative
constant. Furthermore, write $d_{\cal D\cal G}$ to denote 
the distance in the disk graph, and let $d_{{\cal A}(F_0^+)}$ be the distance in the arc graph of 
$F_0^+$. Let $p_1$  be the spot contained in $F_0^+$. 
For two disks $D,E$ whose intersections
with $F_0^+$ contain at least one arc with both endpoints on $c_1$ we write 
$d_{{\cal A}(F_0^+)}(\partial D\cap F_0^+,\partial E\cap F_0^+)$ to denote the distance
in ${\cal A}(F_0^+)$ between any two such arcs. This is coarsely well defined.

\begin{lemma}\label{pointpush2}
Let $D$ be a disk with the property that $\partial D\cap F_0^+$ 
contains an arc with both endpoints on $c_1$ and  
let $E$ be a disk which is obtained from the disk $D$
by point pushing $p_1$ along a loop
in $F_0=F_0^+\cup \{p_1\}$ based at $p_1$. Then 
\[d_{\cal D\cal G}(D,E)\approx d_{{\cal A}(F_0^+)}(\partial D\cap F_0^+,\partial E \cap F_0^+).\]
\end{lemma}
\begin{proof} Let $\Omega_1:{\cal D\cal G}\to {\cal R\cal D}^+(c)$ be  the coarse 
two-Lipschitz retraction
constructed in Lemma \ref{casewithspot} 
from the surface $F_0^+$ (see Remark \ref{notunique}). By assumption,
the intersection with $F_0^+$ of 
$\partial D$ contains an arc $\alpha_0$ 
with both endpoints on $c_1$, and the intersection of $F_0^+$ with $\partial E$ contains
the image $\zeta_0$ of $\alpha_0$ under point pushing along $p_1$. 
Thus we may assume that  
the images of $D,E$ under the map 
$\Omega_1$ are $\Psi$-invariant disks $\hat D,\hat E$ 
which intersect $F_0^+$ in the arcs $\alpha_0$,
$\zeta_0$. By Lemma \ref{casewithspot}, 
we have $d_{\cal C\cal G}(D,E)\geq 
\frac{1}{2}d_{\cal D\cal G}(\hat D,\hat E)-m$ where $m>0$ is a universal constant. 
Since the isometry ${\cal R\cal D}^+(c)\to {\cal A}(F^+)$ associates to 
a disk in ${\cal R\cal D}^+(c)$ the projection of its boundary into $F^+$, this implies 
that $d_{\cal D\cal G}(D,E)$ is not smaller than a fixed positive multiple of 
$d_{{\cal A}(F_0^+)}(\partial D\cap F_0^+,\partial E\cap F_0^+)$. 
Thus it remains
to show that $d_{\cal D\cal G}(D,E)$ is bounded from above by a fixed positive multiple of 
$d_{{\cal A}(F_0^+)}(\alpha_0,\zeta_0)$.

We say that an arc in $F_0^+$ \emph{encloses the spot $p_1$} if it is the boundary of a thickening
of an arc connecting the preferred boundary component $c_1$ 
of $F_0^+$ to the spot $p_1$. 
Any two such arcs which are not isotopic intersect. 
Let $\alpha_i\subset {\cal A}(F_0^+)$ be a geodesic 
connecting $\alpha_0\subset \partial D\cap F_0^+$
to $\alpha_m=\zeta_0\subset \partial E\cap F_0^+$. We begin with replacing this path by a path $\beta_j$
of length at most $2m$ such that for all $j$, the arc $\beta_{2j+1}$ encloses the spot $p_1$.
To this end it suffices to proceed as follows. If $i<m$ is such that both 
$\alpha_i,\alpha_{i+1}$ do not enclose the spot, then as $\alpha_i,\alpha_{i+1}$ are disjoint, there
exists an arc $\hat \beta$ enclosing the spot $p_1$ which is
disjoint from both $\alpha_i,\alpha_{i+1}$. Replace the 
edge in ${\cal A}(F_0^+)$ 
connecting $\alpha_i$ to $\alpha_{i+1}$ by the path of length two with the same endpoints
which passes through $\hat \beta$. Since no two adjacent vertices in the path $\alpha_i$ can 
enclose the spot $p_1$, this yields a path $\beta_j$ as required.

By the analogue of Corollary \ref{arcs}, the arc $\beta_{2j+1}$ is obtained from 
$\beta_{2j-1}$ by point pushing $p_1$ along a loop $\gamma_j\subset F_0^+$ based at $p_1$. 
Since both $\beta_{2j-1}$ and $\beta_{2j+1}$ are disjoint from $\beta_{2j}$, the same holds true for 
$\gamma_j$. By concatenation, for each $j$, the arc $\beta_{2j-1}$ is obtained from 
$\beta_0=\alpha_0$ by point pushing with the loop 
$\hat \gamma_j=\gamma_1\cdot \dots \cdot \gamma_{j-1}$ (read from left to right). As the arc $\beta_{2j}$ 
is disjoint from the image of $\alpha_0$ by point pushing
along $\hat \gamma_{j-1}$, it is the image
under point-pushing along $\hat \gamma_{j-1}$ of an arc $\hat \beta_{2j}$
disjoint from $\alpha_0$. As such an arc
$\hat \beta_{2j}$ is contained in the boundary of a disk
in ${\cal R\cal D}^+(c)$ which is disjoint from $\alpha_0$, 
the arc $\beta_{2j}$ is contained in the boundary of a disk as well, and this disk is disjoint from the 
disk obtained from $D$ by point pushing along $\hat \gamma_{j-1}$.

As a consequence, the path $\beta_j$ is a projection to $F_0^+$ of a path in 
${\cal D\cal G}$ connecting $D$ to $E$. This shows that indeed, 
$d_{\cal D\cal G}(D,E)$ does not exceed $2d_{{\cal A}(F_0^+)}(\alpha_0,\zeta_0)= 
2 d_{{\cal A}(F_0^+)}(\partial D\cap F_0^+,\partial E\cap F_0^+)$.
This completes the proof of the lemma.
\end{proof}

For the formulation of the following corollary, note that the interval in the vertical boundary 
of an $I$-bundle with generator $c$ whose endpoints are the two spots defines a basepoint 
for the arc graph of the annulus $A$ with core curve $c$, and it defines
a disk $D$ enclosing the two spots. 
We call a simple closed curve
$a\subset \partial H$ \emph{untwisted} if the subsurface projection of $a$ into 
this annulus has distance at most two to this basepoint. 
Note that this subsurface projection measures the amount of twisting relative to a base arc
of an arc 
in the annulus $A$ connecting two points in the distinct
boundary components of $A$.

If $c$ is a separating $I$-bundle generator then define ${\cal D\cal G}(c)\subset {\cal D\cal G}$ to 
be the set of all disks which intersect $c$ in precisely two points. 
If $c$ is a non-separating $I$-bundle generator then let 
${\cal D\cal G}(c)\subset {\cal D\cal G}$ be the set of all disks which intersect $c$ in 
precisely two points and are disjoint from the preimage of the boundary of a M\"obius
band in the base of the $I$-bundle.

\begin{corollary}\label{pointpush}
For any $I$-bundle generator $c\subset \partial H$, the inclusion
${\cal D\cal G}(c)\to {\cal D\cal G}$ is a quasi-isometric embeddeding.
Furthermore, 
if $D,D^\prime$ are two such disks whose intersections with the annulus
$A$ with core curve $c$ are untwisted, then 
\[d_{\cal D\cal G}(D,D^\prime)\approx \max\{d_{{\cal A}(F^+)}(\Omega_1(D),\Omega_1(D^\prime)),
d_{{\cal A}(\Psi(F^+))}(\Omega_2(D),\Omega_2(D^\prime))\}\]
where $\Omega_1,\Omega_2$ are the two distinct 
coarse Lipschitz retractions ${\cal D\cal G}\to 
{\cal R\cal D}^+(c)$. 
\end{corollary}
\begin{proof}
Let $\Omega_1,\Omega_2:{\cal D\cal G}\to {\cal R\cal D}^+(c)$ 
be the two coarse Lipschitz retractions defined by the 
choices of the surfaces $F_0^+,\Psi(F_0^+)\subset \partial H$ as before.  
Let $D\in {\cal D\cal G}(c)$ be untwisted. 
We observed in Lemma \ref{pointpush2} and its proof that 
any disk $D_1\in {\cal R\cal D}^+(c)$ whose intersection with the annulus
$A$ is untwisted (or, more generally, untwisted with respect to $D$) and 
which encloses the two spots can be obtained
from $D$ by 
point pushing the point $p_1$ along a loop $\alpha$ in $F_0^+$ and point pushing 
$p_2=\Psi(p_1)$ along a loop $\beta$ in $\Psi(F_0^+)$. 
Since these point pushing operations clearly commute,
the disk $D_1$ does not depend on the order in which these
point pushing transformations are carried out. 

By Lemma \ref{casewithspot}, the distance between
$D_1$ and $D$ is proportional to the maximum of the distance in the arc graph 
${\cal A}(F^+)$ of the projections of the 
intersections $\partial D\cap F_0^+, \partial D_1\cap F_0^+\in {\cal A}(F_0^+)$ and 
$\partial D\cap \Psi(F_0^+),\partial D_1\cap \Psi(F_0^+)\in {\cal A}(\Psi(F_0^+))$.
As the set of disks enclosing the two spots is one-dense in ${\cal D\cal G}(c)$, 
this yields the statement of the corollary.
\end{proof}

\begin{corollary}\label{firststep}
\begin{enumerate}
\item 
If $g=2h$ is even then 
${\cal D\cal G}$ contains quasi-isometrically embedded $\mathbb{Z}^3$.
\item If $g$ is odd then ${\cal D\cal G}$ contains 
quasi-isometrically embedded $\mathbb{Z}^2$.
\end{enumerate}
\end{corollary}
\begin{proof} Let $g=2h$ be even and let 
$c\subset \partial H_0$ be a separating $I$-bundle generator. Using the above notations,
let $\alpha$ be a bi-infinite quasi-geodesic in ${\cal A}(F_0^+)$ 
starting at $\alpha(0)$ where $\alpha(0)$ is the intersection of $F_0^+$ with
the boundary of the disk $D\in {\cal R\cal D}^+(c)$.
Such quasi-geodesics exists since the arc graph ${\cal A}(F_0^+)$ is of
infinite diameter. By what we showed so far, for all $s,t$ there exists a disk 
$\zeta(s,t)$ whose projection to $F_0^+$ equals $\alpha(s)$ and whose
projection to $\Psi(F_0^+)$ is of distance at most two to $\Psi(\alpha(t))$. 
Corollary \ref{pointpush} then shows that 
\[d_{\cal D\cal G}(\zeta(s,t),\zeta(s^\prime,t^\prime))\approx \max\{s-s^\prime,t-t^\prime\}.\]
Thus the image of $\zeta:{\mathbb{Z}}^2\to {\cal D\cal G}$ is uniformly quasi-isometric
to $\mathbb{Z}^2$ equipped with the maximum norm. As this norm is quasi-isometric 
to the standard Euclidean norm, we conclude that 
${\cal D\cal G}$ contains quasi-isometrically embedded $\mathbb{Z}^2$.

Let
$A\subset \partial H_0$ be an annulus neighborhood about $c$.
For a fixed pair of 
points on the boundary of $A$ (say the marked points $p_1,p_2$) the 
curve graph ${\cal C\cal G}(A)$ of $A$  
can be identified with the set of homotopy classes of arcs 
connecting the two boundary components of $A$
where a homotopy is not allowed to 
cross through the points $p_1,p_2$. 
There exists a coarsely well defined
\emph{subsurface projection} $\Pi:{\cal C\cal G}(\partial H)\to 
{\cal C\cal G}(A)$ which maps a simple closed curve crossing through $A$ to 
a component of its intersection with $A$. Here as before, ${\cal C\cal G}(X)$ is the 
curve graph of the surface $X$.

Since $c$ is separating by assumption, 
any diskbounding simple closed curve has an essential intersection with $A$.
Hence the restriction  to ${\cal D\cal G}$ of the subsurface projection 
${\cal C\cal G}(H)\to {\cal C\cal G}(A)=\mathbb{Z}$ 
is a Lipschitz retraction into
$\mathbb{Z}$ which commutes with the two Lipschitz
retractions defined by the surfaces 
$F_0^+$ and $\Psi(F_0^+)$. 
On the other hand, iterated point pushing one of the points $p_i$ about the core curve of the annulus
keeping the second point fixed shows that the image of this projection is all of $\mathbb{Z}$.
Together this shows that there are embedded 
$\mathbb{Z}^3$ in ${\cal D\cal G}$.

Now let $g$ be odd. By Corollary \ref{pointpush} and the 
discussion in the beginning of this proof, we only have to
observe once more that
the diameter of ${\cal A}(F_0^+)\subset {\cal A}(F^+)$
is infinite, which is well known (see \cite{MS10}). 
\end{proof}

\section{Free splittings and sphere graphs}\label{freesplittings} 

This section is devoted to the proof of Theorem \ref{spheregraph}. 
We begin with looking again at a handlebody $H$ of genus
$g\geq 2$ with two spots $p_1,p_2$ on the
boundary. Let $H_0$ be the handlebody of genus $g$ without spots and let
$\Phi:H\to H_0$ be the spot removing map. Recall that 
a disk $D$ in $H$ encloses the two spots
$p_1,p_2$ if $\Phi(D)\subset H_0$ is homotopic to a point. 

We next use the two spots to add a handle to $H$. The resulting
manifold is a handlebody $H^\prime$ of genus $g+1$ with one spot. 
To this end slightly enlarge the two spots $p_1,p_2$ 
to two small compact disjoint disks $B_1,B_2$ in $\partial H$
with $p_i\in \partial B_i$. Identifying these two disks 
with an orientation reversing diffeomorphism $B_1\to B_2$ 
which maps $p_1$ to $p_2$ yields 
a handlebody $H^\prime$ of genus $g+1$. We may view the common image
of the points $p_1,p_2$ as a spot $p\in \partial H^\prime$.
The fundamental group of $H^\prime$ 
is the free group $\mathfrak{F}_{g+1}$ with $g+1$ 
generators. We choose the spot $p$ 
of  $H^\prime$ as the basepoint 
for the fundamental group of $H^\prime$.

The following simple observation will be used several times later on.

\begin{lemma}\label{disksplit}
A disk $D$ in $H$ which encloses the two spots $p_1,p_2$ and the choice
of one of the spots $p_i$ 
determines a free 
splitting $\pi_1(H^\prime,p)=\mathfrak{F}_{g+1}=
\mathfrak{F}_g*\mathbb{Z}$. Changing the spot 
changes the splitting by conjugation with a generator of the 
$\mathbb{Z}$-factor.  
\end{lemma}
\begin{proof}
Up to isotopy, we may assume
that the disk $D$ is disjoint from 
the two closed disks $B_1$ and $B_2$ 
used in the construction of $H^\prime$. Thus
$D$ determines a separating disk $D^\prime$
in $H^\prime$ which only depends on $D$. 
This disk cuts $H^\prime$ into a handlebody of genus
$g$ with fundamental group $\mathfrak{F}_g$ and a solid torus $T$
with fundamental group $\mathbb{Z}$ which 
contains the basepoint $p$.

Van Kampen's theorem now shows that $D^\prime$ 
defines a free splitting 
\[\pi_1(H^\prime,p)=\mathfrak{F}_{g+1}=
\mathfrak{F}_g*\mathbb{Z},\] 
unique up to conjugation with an element of the free factor $\mathbb{Z}$.
Namely, 
the basepoint $p$ is contained in the solid torus $T$.  
Thus the splitting of $\pi_1(H^\prime,p)$ obtained by
van Kampen's theorem is determined by $D^\prime$ up to conjugation with
an element of $\pi_1(T)$. 

To see that if we fix one of the spots $p_i$ then 
we obtain in fact a uniquely determined splitting, it suffices
to observe that the solid torus $T$ is obtained by identifying two disks in 
the boundary of a ball. This ball is fixed, but the disks are allowed to 
move within a fixed subdisk $D$ of this boundary. 
As a disk is contractible, moving the two disks $B_1,B_2$ 
freely in $D$ 
gives rise to the same splitting and hence there is no ambiguity in the 
construction (in other words, the fundamental group of the solid torus $T$ appears only
after the gluing). 
\end{proof}

The construction in Lemma \ref{disksplit} can be reversed. Namely, observe that
the handlebody $H^\prime$ contains a distinguished non-separating disk $V$ which is 
the image of the two disks in $\partial H$ used in the construction. The spot of $H^\prime$
is contained in the boundary of $V$.
If $E\subset H^\prime$
is any separating disk disjoint from $V$ which decomposes $H^\prime$ into
a solid torus $T\supset V$ and a handlebody of genus $g$, then 
$E$ is the image of a disk in $H$ enclosing the two spots 
under the gluing construction.

\begin{remark}\label{notinjective}
By Lemma \ref{disksplit},
each disk $D$ in $H$ enclosing the two spots defines a
free splitting $\mathfrak{F}_{g+1}=
\mathfrak{F}_g*\mathbb{Z}$. Here the 
$\mathfrak{F}_g$-factor in the free product is identified with the 
fundamental group $\pi_1(H\cup \{p_2\},p_2)$. 
Lemma \ref{pointpush1} immediately implies the following. 
Let $a\in \mathfrak{F}_{g+1}$ be the generator
of the $\mathbb{Z}$-factor in the 
free splitting of $\mathfrak{F}_{g+1}$ 
defined by the disk $D$ and the choice of the basepoint $p_1$, 
where $a$ is 
viewed as a homotopy class with fixed endpoints of an arc connecting 
$p_1$ to $p_2$. Let $E$ be the image of $D$ by point pushing $p_2$ along
a loop in the homotopy class $q(D,E)\in \pi_1(\partial H_0,p_2)$ defined as the 
class of the concatenation of an arc $\alpha$ connecting $p_2$ to $p_1$ which is 
disjoint from $D$ with an arc $\beta$ connecting $p_2$ to $p_1$ which is disjoint 
from $E$ (compare 
Lemma \ref{pointpush1}). Via the inclusion $\partial H_0\to H_0$, 
this homotopy class  defines
a homotopy class $\iota_*q(D,E)\in \pi_1(H\cup \{p_2\},p_2)=\pi_1(H_0,p_2)$
(here the last equation is the identification of fundamental groups under the
spot closing map). Then the 
$\mathbb{Z}$-factor defined by $p_1$ and the disk $E$ is generated by 
$a\cdot \iota_*q(D,E)$ (read from left to right).
\end{remark}

From now on we fix a disk $D$ enclosing the two spots in $H$ which is a thickening of 
an interval in an $I$-bundle over a compact surface $F$ with connected boundary.
If the genus of $H$ is even then we assume that $F$ is orientable. 
This disk defines a 
free splitting $\mathfrak{F}_{g+1}=\mathfrak{F}_g*\mathbb{Z}$ where
the free factor $\mathbb{Z}$ is generated by an element $a$ 
obtained from an embedded oriented arc in the twice spotted
disk $\tilde D$ in $\partial H$ 
with the same boundary as $D$ 
which connects $p_1$ to $p_2$. 

Double the handlebody $H$ to a connected sum $M$ 
of $g$ copies of $S^1\times S^2$ with two spots. 
This defines an embedding $\mathfrak{J}: H\to M$. Any disk $E$ in 
$H$ doubles to an essential sphere $\Pi(E)\subset M$, and $E$ encloses the two
spots if and only if this holds true for $\Pi(E)$.  
Furthermore, disjoint disks give rise to disjoint spheres. Hence the doubling map 
$\Pi:{\cal D\cal G}\to {\cal S\cal G}$ is simplicial, where ${\cal S\cal G}$ 
is the sphere graph of $M$. 
The following observation shows that the map $\Pi$ is not bilipschitz onto its image.

\begin{lemma}\label{projectrd}
For each $I$-bundle generator $c\subset \partial H$, 
the image of the subgraph ${\cal R\cal D}^+(c)$ under the map $\Pi$
has diameter at most two   
in the sphere graph ${\cal S\cal G}$.
\end{lemma}
\begin{proof} As any disk in ${\cal R\cal D}^+(c)$ is at distance one
from a disk $E\in {\cal R\cal D}^+(c)$ enclosing the two spots, 
it suffices to show that two disks $D,E\in {\cal R\cal D}^+(c)$ enclosing the two spots
are mapped by $\Pi$ to the same sphere enclosing the two spots. 

A disk $E\in {\cal R\cal D}^+(c)$ enclosing the two spots is 
the image of the base disk $D$ by point pushing the point $p_2$ along
 a based loop $\zeta\in \pi_1(F_0,p_2)$,
 followed by point pushing $p_1$ along the based loop 
 $\Psi(\zeta)\in \pi_1(\Psi(F_0),p_1)$. 
 
 Let $\alpha$ be an arc connecting 
 $p_1$ to $p_2$ which is disjoint from $D$.
 Lemma \ref{pointpush1},
 applied to both the point pushing of $p_1$ and of $p_2$, shows that up to 
 an ambiguity arising from clearing intersections with the spots,  an arc 
 connecting $p_1$ to $p_2$ which is disjoint from $E$ 
 is homotopic with fixed endpoints
 to $\Psi(\zeta)^{-1} \circ \alpha \circ \zeta$ (read from left to right).
 
Note that the reflection $\Psi$ acts as the 
identity on the fundamental group of $H$, taken at a fixed point for $\Psi$. 
Thus the loop obtained by connecting
the fixed point $q$ for $\Psi$ in $\alpha$ to $p_2$, concatenating with 
$\zeta$ and going back to $q$ along $\alpha$ is homotopic to 
its image under $\Psi$. Now a sphere in the doubled handlebody $M$ 
enclosing the two spots is the boundary of the thickening of an arc connecting
the two spots. By the above discussion, the arc defining the sphere $\Pi(E)$ 
is homotopic to the arc which defines the sphere $\Pi(D)$.
Informally, a sphere $S$ enclosing the two spots $p_1,p_2$ determines uniquely
an isomorphism of the fundamental group of $M$ based at $p_1$ with 
the fundamental group of $M$ based at $p_2$ by connecting $p_1$ to $p_2$ with 
an arc not crossing through $S$. With this identification, point pushing $S$ along
loops at $p_1,p_2$ defining the same element in the fundamental group of $M$ 
gives rise to the same sphere.
Note that the loops $\zeta$ and $\alpha^{-1}\circ \psi(\zeta)\circ \alpha$ 
are contained in the boundary $\partial H$ of 
the handlebody, and they are not homotopic as elements of the fundamental group 
of $\partial H$.  
This shows the  lemma.
\end{proof}


We shall construct spheres as doubles of disks in the handlebody $H$ and keep track of
distances using Lemma \ref{projectrd}. Let again $D$ be a disk 
which is the thickening of an interval
in the $I$-bundle defined by the $I$-bundle generator $c$. It defines the generator 
$a$ of the $\mathbb{Z}$-factor of the splitting 
$\mathfrak{F}_{g+1}=\mathfrak{F}_g*\mathbb{Z}$ defined by $D$.

Let $d_{\cal S\cal G}$ be the distance in the sphere graph. For ease of bookkeeping, 
define a function $d_{\widehat{\cal D\cal G}}$ on pairs of disks in 
$H$ by 
%
\[d_{\widehat{\cal D\cal G}}(E,F)=d_{\cal S\cal G}(\Pi(E),\Pi(F)).\]
Note that $d_{\widehat{\cal D\cal G}}$ is symmetric and fulfills the triangle inequality, but 
by Lemma \ref{projectrd} and by Lemma \ref{casewithspot}, it is 
\emph{not} comparable to the distance
on ${\cal D\cal G}$. 
Furthermore, we have
$d_{\widehat{\cal D\cal G}}(E,F)\leq d_{\cal D\cal G}(E,F)$ for all disks $E,F$ since the 
map $\Pi$ is one-Lipschitz.

Recall that point pushing one of the two spots in $H$ is a diffeomorphism of $H$ which 
extends to a diffeomorphism of $M$ and hence induces an isometry on the sphere graph 
${\cal S\cal G}$. In other words, such a map preserves the function 
$d_{\widehat{\cal D\cal G}}$. Informally, we say that such a map is an isometry for 
$d_{\widehat{\cal D\cal G}}$.

Consider as before an embedded 
subsurface $F_0$ of $\partial H_0$ determined by an 
$I$-bundle generator $c$. If $g$ is even then we assume that 
$c$ is separating and $F_0$ is isotopic to a component of $\partial H_0-c$. 
If $g$ is odd then $F_0$ is a component of the complement of the preimage
of the core curve of a M\"obius band in the orientation cover of the 
non-orientable base of a non-separating $I$-bundle generator $c\subset \partial H_0$.
Let $\Psi$ be the orientation reversing involution of $H_0$ determined by the $I$-bundle generator
$c$. Let $p_2\in F_0$ be a point in the interior of $F_0$ and let $p_1=\Psi(p_2)$. 
Let $D \in {\cal R\cal D}^+(c)$ be a disk isotopic to the thickening of 
a fiber arc of the $I$-bundle with endpoints $p_1,p_2$, thought of as being 
constructed by slightly pushing a point on $\partial F_0$ inside $F_0$.

 Let us denote by $[a,u]_2$ the disk obtained from $D$ by point pushing $p_2$ along a loop
$\gamma$ based at $p_2$ in the surface $F_0$, 
in the homotopy class $u\in \pi_1(F_0,p_2)$, and by $[a^{-1},u]_1$ the disk obtained from 
$D$ by point pushing $p_1$ along the based loop $\Psi(\gamma)\subset \Psi(F)$
in the homotopy class $\Psi(u)$. 
With this
notation, for all $u\in \pi_1(F_0,p_2)$ we have 
\begin{equation}\label{pointcommute}
d_{\widehat{\cal D\cal G}}([a,u]_2,[a^{-1},u^{-1}]_1)\leq 2.\end{equation}
Namely, point pushing 
$p_1$ along the loop $\Psi(\gamma)$ 
is an isometry for $d_{\widehat{\cal D\cal G}}$,  and the
image of the disk $[a,u]_2$ under this point pushing map 
is a disk which is invariant under $\Psi$ and hence contained in the
subspace ${\cal R\cal D}^+(c)$. The image of this space under the 
projection map $\Pi$ has diameter two by Lemma \ref{projectrd}.

Now if we write composition from left to right, then as point pushing along based loops at $p_1$ 
preserves $d_{\widehat{\cal D\cal G}}$, 
for homotopy classes $b_1,b_2,c\in \pi_1(F_0,p_2)$ we obtain from (\ref{pointcommute}) that 
\begin{align}\label{crucial2}
d_{\widehat{\cal D\cal G}}([a,c\cdot b_1]_2,[a,c\cdot b_2]_2)\leq 
d_{\widehat{\cal D\cal G}}([a^{-1},b_1^{-1}\cdot c^{-1}]_1,[a^{-1},b_2^{-1}\cdot c^{-1}]_1)+4\\
=d_{\widehat{\cal D\cal G}}([a^{-1},b_1^{-1}]_1,[a^{-1},b_2^{-1}]_1)+4
\leq d_{\widehat{\cal D\cal G}}([a,b_1]_2,[a,b_2]_2)+8.\notag\end{align}


The free factor $\mathfrak{F}_g$ in the free splitting $\mathfrak{F}_{g+1}=\mathfrak{F}_g*\mathbb{Z}$ 
defined by a disk enclosing the two spots 
is naturally isomorphic to 
$\pi_1(H\cup p_2,p_2)$. Thus a free basis 
${\cal A}=\{a_1,\dots,a_g\}$ of $\mathfrak{F}_{g}=\pi_1(H\cup p_2,p_2)$ extends to a free basis
$\hat {\cal A}=\{a_1,\dots,a_g,a\}$ 
of $\mathfrak{F}_{g+1}$.

We now use a device from \cite{SS14}. 
Define the \emph{Whitehead graph} $\Gamma_{\cal A}(x)$ of 
a word $x\in \mathfrak{F}_g$ in a free basis ${\cal A}\cup 
{\cal A}^{-1}$ of $\mathfrak{F}_g$ as follows.
The set of vertices of $\Gamma_{{\cal A}}(x)$ is
identified with the set ${\cal A}\cup {\cal A}^{-1}$.
Each pair of consecutive letters $a_ia_j$ in the 
word $x$ contributes one edge 
from the vertex $a_i$ to the vertex $a_j^{-1}$. Thus if the 
length of $x$ equals $n$ then $\Gamma_{\cal A}(x)$ has $n-1$ edges,
and $\Gamma_{\cal A}(x)$ has a cut vertex if $x\in {\cal A}$. 
Furthermore, if $\Gamma_{\cal A}(x)$ has a cut vertex, 
then the same holds true for the unique reduced word
which defines the same element of $\mathfrak{F}_g$ as $x$. 

Following \cite{SS14}, 
define the \emph{simple $g+1$-length} 
\[\vert w\vert_{g+1}^{simple}\] 
of any reduced word $w$ in the free basis 
${\cal A}=\{a_1,\dots,a_g\}$ of $\mathfrak{F}_g$ 
to be the greatest number $t$ 
such that $w$ is of the form $w_1w_2\cdots w_t$ where the
Whitehead graph of $w_j$ with respect to the basis ${\cal A}$ 
has no cut vertex for each $j=1,\dots,t$. If the Whitehead
graph of $w$ has a cut vertex then the simple $g+1$-length
of $w$ is defined to be zero. We have that $\vert w\vert_{g+1}^{simple}$
is bounded from above by the word length of 
the reduced word $w$ with respect to the
basis ${\cal A}$. Furthermore, $\vert w^{-1}\vert_{g+1}^{simple}=
\vert w\vert_{g+1}^{simple}$.
The terminology here is taken
from \cite{SS14} although it is not well adapted to the situation at
hand. 

The following statement combines Lemma 4.6 and Lemma 4.7 of 
\cite{SS14},

\begin{lemma}\label{sabalka}
\begin{enumerate}
\item $\vert u\vert_{g+1}^{simple}\geq \vert v\vert_{g+1}^{simple}$
whenever $v$ is a subword of $u$.
\item 
\[\vert w\vert_{g+1}^{simple}\leq \vert u\vert_{g+1}^{simple}+\vert v\vert_{g+1}^{simple} +1\]
if $u,v$ are freely reduced words in the letters ${\cal A}\cup {\cal A}^{-1}$ and $w=uv$.
\end{enumerate}
\end{lemma}
\begin{proof}
The statement of Lemma 4.7 of \cite{SS14} shows the second part of the lemma only in the
case that $w=uv$ is freely reduced. To show that it is true as stated, assume that 
$\vert v\vert_{g+1}^{simple}=0$ and that $w$ is the reduced word representing $uv$.
Then $w$ is obtained from $uv$ by 
erasing some letters at the end of $u$ and the beginning of $v$. In particular, 
by the first part of the lemma, the 
Whitehead graph of the subword of $v$ 
which is contained in $w$ has a cut vertex. 

As a consequence,  if 
$w=w_1\cdots w_t$ where the Whitehead graph of 
$w_i$ does not have a cut vertex, then as $u$ is reduced, 
$w_1\cdots w_{t-1}$ is a subword of $u$. Then 
$t-1\leq \vert u\vert_{g+1}^{simple}$ by the first part of the lemma and hence
$\vert w\vert_{g+1}^{simple}\leq \vert u\vert_{g+1}^{simple}+1$ as claimed.

The general case follows from a rather straightforward modification of this argument
and will be omitted. Only the case that $\vert v\vert_{g+1}^{simple}=0$ is used in the sequel.
\end{proof}

The next lemma relates simple $g+1$-length to the sphere graph
${\cal S\cal G}$ of $M$. To simplify the notation,  
in the sequel we call a sequence $(S_i)$ of spheres in $M$
a \emph{path} in ${\cal S\cal G}$ if for all $i$ the sphere $S_i$ is disjoint
from $S_{i+1}$. Thus such a sequence is the set of integral points on 
a simplicial path in ${\cal S\cal G}$ connecting its endpoints.
Recall from Lemma \ref{pointpush1} and its analogue
for spheres that an ordered pair $(S,U)$ of spheres enclosing 
the two spots $p_1,p_2$ 
determines uniquely an element $b(S,U)\in \pi_1(M,p_2)$, which is 
represented by the concatenation of an arc connecting $p_2$ to 
$p_1$ not crossing through $S$ with an arc connecting 
$p_1$ to $p_2$ not crossing through $U$.

\begin{lemma}\label{simplelength}
Let $(S_i)_{0\leq i\leq n}$ be a path in ${\cal S\cal G}$ which begins
and ends with a sphere enclosing the two spots $p_1,p_2$.
Let $w= b(S_0,S_n)\in \pi_1(M,p_2)$;
then 
\[ \vert w\vert_{g+1}^{simple}\leq 2n.\]
\end{lemma} 
\begin{proof} Assume without loss of generality that
the path $(S_i)$ connecting $S_0$ to $S_n$ is 
of minimal length in ${\cal S\cal G}$. 
First we modify inductively the sequence $(S_i)$ without
increasing its length in such a way that each of the spheres
$S_i$ $(1\leq i\leq n-1)$ either is non-separating
or encloses the spots $p_1,p_2$.

The construction proceeds in two steps. 
In a first step, we replace each separating sphere $S_{2i-1}$ with 
odd index by a sphere which either is non-separating or
encloses the two spots. We do not change the spheres 
$S_{2i}$ with even index.
In a second step, we then modify the spheres with even index
and preserve those with odd index.

To carry out the first step, 
let $\ell\leq n/2$ and assume that the sphere
$S_{2\ell-1}$ is separating 
and does not enclose the spots; otherwise there is
nothing to do.
If $S_{2\ell-2},S_{2\ell}$  are contained in 
distinct components of $M-S_{2\ell-1}$ then they are 
disjoint. In this case we can remove $S_{2\ell-1}$ 
from the path $(S_i)$ and obtain a shorter path with the same
endpoints. Since the path $(S_i)$ has minimal length
this is impossible. 

Thus $S_{2\ell-2},S_{2\ell}$ are contained
in the same component $W$ of $M-S_{2\ell-1}$. Since 
$S_{2\ell-1}$ does not enclose the spots, 
neither of the two components of $M-S_{2\ell-1}$ 
is a ball with two balls (or points) removed from the interior. 
Since $M$ has precisely two spots, this implies that 
the image of the fundamental group of 
each of the two components of 
$M-S_{2\ell-1}$ in the fundamental group of $M$ is non-trivial.
Now each component of 
$M-S_{2\ell-1}$ is a connected sum of $S^1\times S^2$ with 
some balls removed and therefore 
the component $M-W$ 
of $M-S_{2\ell-1}$ contains a non-separating sphere
$\tilde S_{2\ell-1}$. 
Replace $S_{2\ell-1}$ by $\tilde S_{2\ell-1}$.

Replace in this way any sphere
$S_{2\ell-1}$ with an odd index which is separating but
does not enclose $p_1,p_2$  
by a non-separating sphere
without modifying the spheres $S_{2i}$ 
with even index. This implements the first step of 
the construction. The second step is exactly identical 
after exchanging the roles 
of even and odd index.
To summarize, we may assume from now on that 
every separating sphere in 
the path $(S_i)$  
encloses the two spots.

From the path $(S_i)$ we next construct a path 
$(U_k)_{0\leq k\leq 2u}$ 
of spheres connecting $S_0$ to $S_n$ 
whose length $2u$ is at most four times the 
length $n$ of the path $(S_i)$ and such that  
for each 
$j\leq u$, the sphere $U_{2j}$ encloses the spots $p_1,p_2$
and the sphere $U_{2j-1}$ is non-separating.

To this end recall that any two distinct spheres which 
enclose the spots $p_1,p_2$ intersect. This means that
if the sphere $S_i$ from the above sequence
encloses the spots then the spheres
$S_{i-1},S_{i+1}$ are non-separating. 
Thus for the construction of the
path $(U_k)$ it now suffices to replace any consecutive pair $S_i,
S_{i+1}$ of disjoint non-separating spheres 
by a path of length at most four with the same endpoints
whose vertices alternate between non-separating spheres and
spheres enclosing 
the spots.

Let $i<n-1$ be such that the 
spheres $S_i,S_{i+1}$ are both
non-separating. If $M-(S_i\cup S_{i+1})$ is connected then 
there is a sphere $B$ which 
encloses the spots $p_1,p_2$ and which is disjoint from 
$S_i\cup S_{i+1}$. Such a sphere can be obtained by thickening
an  arc in $M-(S_i\cup S_{i+1})$
which connects $p_1$ to $p_2$. 
Replace the consecutive pair $S_i,S_{i+1}$ by
the path $S_i,B,S_{i+1}$ of length two.

If $M-(S_i\cup S_{i+1})$ is disconnected and if the spots 
$p_1,p_2$ are both contained in the same component of 
$M-(S_i\cup S_{i+1})$, then we can proceed as in the previous
paragraph. Otherwise there is a component of 
$M-(S_i\cup S_{i+1})$ which is a connected sum of 
$h\geq 1$ copies of $S^1\times S^2$ with three  
points or open balls removed. 
One of the holes
is a spot, the other two holes are bounded by spheres 
which glue to 
the spheres $S_i,S_{i+1}$.
Hence  there is a non-separating
sphere $B$ which is disjoint from 
$S_i\cup S_{i+1}$ and such that $M-(S_i\cup B)$ and 
$M-(S_{i+1}\cup B)$ are both connected. 
Replace the consecutive pair $S_i,S_{i+1}$ by a path 
of length $4$ of the form $S_i,A_1,B,A_2,S_{i+1}$ 
where the spheres $A_1,A_2$ both enclose
the spots $p_1,p_2$. This completes the construction 
of the sequence $(U_k)$.

For each $j$ the sphere $U_{2j}$ defines a free splitting
of $\mathfrak{F}_{g+1}$ of the form $\mathfrak{F}_g*\mathbb{Z}$. 
If $a$ is the generator of the free factor $\mathbb{Z}$ for 
the splitting defined by $S_0$ (in the sense as before, namely we 
think of $a$ as a homotopy class of an arc in 
$M$ which connects $p_1$ to $p_2$ and does not intersect the sphere
$S_0$, and this homotopy class determines the free
factor $\mathbb{Z}$ in the free splitting defined by $S_0$), 
then for each $j$ 
the free factor $\mathbb{Z}$ for the free splitting defined
by $U_{2j}$ is generated by $a\cdot b(S_0,U_{2j})$
where $b(S_0,U_{2j})\in \mathfrak{F}_g$. 

Let $w_j=b(S_0,U_{2j})^{-1} b(S_0,U_{2j+2})\in \mathfrak{F}_g$. 
By construction, the sphere $U_{2j+1}$ in $H$ is non-separating
and disjoint from $U_{2j},U_{2j+2}$. 
The set of all loops in $M\cup \{ p_2\}$ with basepoint $p_2$ 
which do not intersect $U_{2j+1}$ defines a free factor 
$Q$ of $\mathfrak{F}_g$ of corank one. 
Since $U_{2j+1}$ is disjoint
from $U_{2j}$ and $U_{2j+2}$, the element $w_j$ is contained in
the free factor $Q$.

Since $w_j\in Q$, 
by Theorem 2.4 of \cite{S00} the Whitehead graph of 
$w_j$ has a cut vertex. 
But this just means that the simple $g+1$-length of $w_j$ 
vanishes. An inductive application of 
the second part of Lemma \ref{sabalka} now shows that 
the simple $g+1$-length of the word 
$b(S_0,S_n)\in \mathfrak{F}_g$ is at most $u\leq 2n$.
This is what we wanted to show.
\end{proof}

Now we are ready to show Theorem \ref{diskgraph}
from the introduction.

\begin{theorem}\label{infiniteasymp}
For every $g\geq 4$ and every $n\geq 1$,  
the sphere graph of a connected sum 
$\sharp_g S^1\times S^2$ with two spots 
contains quasi-isometrically
embedded copies of $\mathbb{R}^n$.
\end{theorem}

Before we provide the proof, we outline the strategy for the argument
which is adapted from \cite{SS14}. 
View as before the manifold $M$ as the double of a handlebody $H$, and 
represent the handlebody $H$ as an $I$-bundle over a compact surface
$F$ with boundary. A sphere $E$ in $M$ 
 enclosing the two spots is obtained from 
a base sphere $S$ enclosing the two spots by point pushing the spot
$p_2$ along a loop in $F$, based at $p_2$.

Point pushing induces an isometry of the sphere graph, which yields that for any 
homotopy classes $u,b_1,b_2\in \pi_1(F,p_2)$ we have
\[d_{\cal S\cal G}([a,b_1]_2,[a,b_2]_2)=
d_{\cal S\cal G}([a,b_1 \cdot u]_2,[a,b_2\cdot u]_2)\]
(recall that we write concatenation from the left to the right). On the 
other hand, the estimate (\ref{crucial2}) shows that a similar
relation holds for \emph{precomposition} with a fixed homotopy class
of a based loop in $(F,p_2)$, which distinguishes the sphere graph of $M$
from the disk graph of the handlebody $H$.
This allows to estimate from above 
the distance in ${\cal S\cal G}$ of two spheres enclosing the two spots
which are obtained by point pushing the base sphere enclosing the two spots 
along suitably chosen loops. 
A lower bound for such distances was established in Lemma \ref{simplelength}.
An explicit construction of point pushing loops, taken from 
\cite{SS14},  whose effect on the base sphere 
can be controlled using this mechanism
leads to the proof of the theorem.

\begin{proof}[Proof of Theorem \ref{infiniteasymp}]
 As before, we view $M$ as the double of the handlebody 
$H$. Assume first that the genus $g=2h$ of $H$ is even.
We will explain at the end of this proof how to adjust the argument
to the case that $g$ is odd.

Let $F\subset \partial H_0$ be an embedded oriented surface with 
connected boundary $\partial F$ such that $H_0$ equals the
$I$-bundle over $F$. Let $\Psi$ be the orientation reversing involution of 
$H_0$ which exchanges the endpoints of the intervals which make up the interval
bundle. 
Fix a point $p_2\in \partial F$ and let $p_1=\Psi(p_2)$.

Arrange the $h=g/2\geq 2$ handles of the surface $F$ 
cyclically around $\partial F$. 
Choose for each handle of $F$ two oriented disjoint non-homotopic 
essential arcs in the handle 
with endpoints on $\partial F$. 
We may assume that $\partial F$ is partitioned into $h$ segments 
$I_1,\dots, I_h$ with disjoint interior, ordered cyclically along $\partial F$ (that is,  
if $h\geq 3$ then $I_j\cap I_{j+1}$ consists of a single point for all $j$)
so that each of these segments $I_j$ is contained in the boundary of 
one of the handles and 
contains all four endpoints of 
the arcs $\hat a_{2j-1}, \hat a_{2j}$ which are embedded in that 
handle. The figure shows how this can be done.
\begin{figure}
\begin{center}
\includegraphics[width=0.6\textwidth]{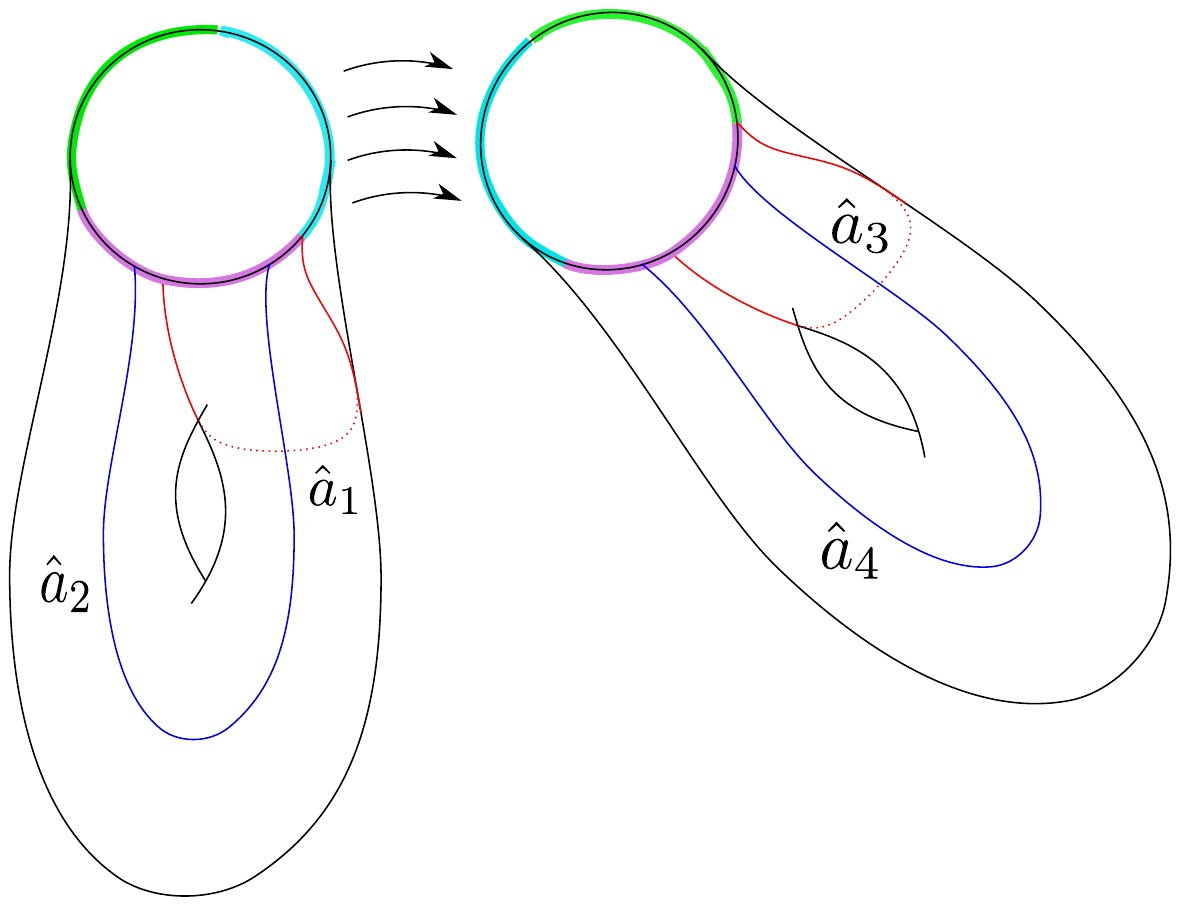}
\end{center}
\end{figure} 

A small neighborhood of the union of these $2h$ arcs 
and the boundary of $F$ is a ribbon graph, 
that is, a planar surface $F_0\subset F$. We require that the inclusion $F_0\to F$ induces a surjection on
fundamental groups. This is equivalent to stating that 
$F$ can be obtained from $F_0$ by attaching a disk
to each component of the boundary of $F_0$ distinct from $\partial F$. 

If $h\geq 3$ then let 
$p_2$ be the intersection
$I_{h}\cap I_1$ and let $x=I_{h-1}\cap I_h$. If $h=2$ then we require that
$\{p_2,x\}=I_1\cap I_2$.
Slide the endpoints of the arcs $\hat a_i$ which define the ribbon graph $F_0$ 
along $\partial F$ to $p_2$ in such a way that this sliding operation 
does not cross through $x$. The image of each of the arcs $\hat a_i$ under this
homotopy is a based oriented loop $a_i$ at $p_2$. The union of these loops is 
an embedded rose $R$ with vertex $p_2$ (the rose $R$ does not contain the boundary circle of $F$).
As $H_0$ is an $I$-bundle over $F$, the inclusion 
$R\to H\cup \{p_2\}$ induces an isomorphism of $Q=\pi_1(R,p_2)$ 
onto the group 
$\pi_1(H\cup \{p_2\},p_2)$ which is isomorphic to the fundamental group of $H$.
Thus if we write $H_2=H\cup \{p_2\}$ then 
we have $\pi_1(H_2,p_2)=Q$. 
In the sequel we think of the based loops $a_i$ $(i=1,\dots,2h)$ as generators of the 
fundamental group $Q$ of $R$. 

As on p.592 in  
Subsection 5.2 of \cite{SS14}, we consider for $t\geq 1$ the element
\[b_t=a_1^{t+1}a_2^{t+1}\cdots a_g^{t+1}a_1^{t+1}a_2^{t+1}a_1^{t+1}\in Q.\]
Let $D$ be the disk in $H$ enclosing the two spots which is a thickening
of the fiber of the interval bundle with endpoints $p_1,p_2$.
We claim that for every $t\geq 1$ 
the image of $D$ under the point 
pushing map of $p_2$ along $b_t$ has distance at most 6 to $D$ in the 
disk graph ${\cal D\cal G}$ of $H$.

We show the claim first in the case that the genus $g=2h$ of $H$ is at least 
$6$ and hence the genus of $F$ is at least three. Then $b_t=uv$ where 
$u=a_1^{t+1}\cdots a_{g-2}^{t+1}$ 
and $v=a_{g-1}^{t+1}a_g^{t+1}a_1^{t+1}a_2^{t+1}a_1^{t+1}$.  
The word $u$ 
does not contain the letters $a_{g-1},a_{g-1}^{-1},a_g,a_g^{-1}$,
and the word $v$ does not contain
the letters $a_{g-3},a_{g-3}^{-1},a_{g-2},a_{g-2}^{-1}$ since $g-3\geq 3$. 

As a consequence, the word $u$ is represented by a loop
in the rose $R$ whose image in the ribbon graph $F_0$ is disjoint from 
the arcs with endpoints in $I_{h}$. 
Hence up to homotopy, this loop is disjoint from the $I$-bundle over 
each of these two arcs. Then the same holds true 
for the image 
$\psi_u(D)$ of the disk $D$ under the point pushing 
map $\psi_u$ along $u$. 
In particular, the distance between $D$ and $\psi_u(D)$ in 
the disk graph ${\cal D\cal G}$ is at most two
(see Lemma \ref{pointpush2}). 
Similarly, the image $\psi_v(D)$ of $D$ 
under the point pushing
map $\psi_v$ along $v$ 
is disjoint from an $I$-bundle over an arc 
with endpoints in the interval $I_{h-1}$
and hence 
 $d_{\cal D\cal G}(D,\psi_v(D))\leq 2$.
But the point pushing map $\psi_v$ acts on the disk graph 
as a simplicial isometry and consequently 
$d_{\cal D\cal G}(\psi_v(D),\psi_v(\psi_u(D)))\leq 2$. 
As $b_t=uv$, together with the triangle inequality  this yields 
\[d_{\cal D\cal G}(D,\psi_{b_t}(D))\leq 4\]
(here words are read from left to right). 

If $g=4$ then write $b_t=uvw$ where $u=a_1^{t+1}a_2^{t+1}$, 
where $v=a_3^{t+1}a_4^{t+1}$ 
and $w=a_1^{t+1}a_2^{t+1}a_1^{t+1}$.
Then there is a loop in $R$ representing $u,v,w$ 
which is disjoint from an arc with endpoints  in $I_2$, $I_1$, $I_2$.   
As in the previous paragraph,  
we conclude that $d_{\cal D\cal G}(D,\psi_s(D))\leq 2$
for $s=u,v,w$. Thus by the triangle inequality, we have 
$d_{\cal C\cal G}(D,\psi_{b_t}(D))\leq 6$. 

This argument can be used inductively and shows the following.
For all $t\geq 1$ and each $k\geq 1$, we have
\begin{equation}\label{distanceestimate}
d_{\cal D\cal G}(D,\psi_{b_t^k}(D))\leq 6k.\end{equation}

Recall that the disk $D$ defines a 
free splitting $\pi_1(H^\prime,p)=\mathfrak{F}_g*\mathbb{Z}$
where as before, $p\in \partial H^\prime$
is the point obtained 
by identification of $p_1$ and $p_2$. 
Let $a$ be the generator of the infinite cyclic group $\mathbb{Z}$,
defined by the homotopy class of 
the arc $\alpha$ in $\partial H$ connecting $p_2$ to $p_1$ which is disjoint from
the boundary of $D$. As explained
in the discussion preceding Remark \ref{notinjective}, 
if $u\in Q$ is arbitrary, then the image of $D$ under the
point-pushing map $\psi_u$ is a disk $\psi_u(D)$ 
enclosing the two spots which 
defines the free splitting
of $\mathfrak{F}_{g+1}$ 
where the infinite cyclic free factor in the splitting 
is generated by $a\cdot \iota_*q(D,\psi_u(D))$. 
By the definition of the point pushing 
map, if we identify $Q$ with $\pi_1(H,p_2)=
\mathfrak{F}_g<\pi_1(H^\prime,p)$ as
described in the beginning of this proof, the generator of 
this infinite cyclic free
factor is just the element $au$. We refer to the discussion before
Lemma \ref{disksplit} for more details.

Using the
above notations, we follow  Section 5.2 of 
\cite{SS14}.  For an arbitrary integer $n\geq 1$, define  
a map $\Lambda: \mathbb{Z}^n\to {\cal D\cal G}$ 
which associates to $(k_1,\dots, k_n)\in \mathbb{Z}^n$
the image of the disk $D$ under point-pushing of $p_2$  
along the loop 
$b_1^{k_1}b_2^{k_2}\cdots b_n^{k_n}\in Q$
based at $p_2$.
We claim that 
\begin{equation}\label{lower}
d_{\widehat{\cal D\cal G}}(\Lambda(k_1,\dots,k_n),\Lambda(\ell_1,\dots,\ell_n))\leq 
6 \sum_{i=1}^n (\vert k_i-\ell_i\vert +8).\end{equation}

To see this we adapt an argument from  
p.594 of \cite{SS14}. Our goal is to transform the disk
$\Lambda(k_1,\dots,k_n)=\psi_{b_1^{k_1}\cdots b_n^{k_n}}(D)$ to the disk 
$\Lambda(\ell_1,\dots,\ell_n)=\psi_{b_1^{\ell_1}\cdots b_n^{\ell_n}}(D)$ 
in a controlled way. These 
disks are determined by the homotopy classes 
$b_1^{k_1}b_2^{k_2}\cdots b_n^{k_n}\in Q$ 
and $b_1^{\ell_1}b_2^{\ell_2}\cdots b_n^{\ell_n}\in Q$, respectively, 
provided that the 
base disk $D$ is fixed. To take full advantage of this
fact we will now consider pairs of disks $(E,V)$ where we
view $V$ as a basepoint, and $E$ as a modification of the 
basepoint. With this viewpoint, our goal will be to transform  
the pair $(\psi_{b_1^{k_1}\cdots b_n^{k_n}}(D),D)$
to the pair $(\psi_{b_1^{\ell_1}\cdots b_n^{\ell_n}}(D),D)$
in a way which enables us to estimate the function 
$d_{\widehat{\cal D\cal G}}$. 

To simplify the discussion, let us resume 
the following notation.
For an element $u\in Q$, represented up to homotopy by a unique
reduced edge path in the rose $R$, let us denote by
$[a \cdot u]_2$ the disk $\psi_u(D)$ obtained from $D$ by point pushing $p_2$ 
along $u$, and denote by $[a^{-1}\cdot u]_1$ the disk obtained from 
$D$ by point pushing $p_1=\Psi(p_2)$ along $\Psi(u)$. 
Inequality (\ref{pointcommute}) 
shows that 
\begin{equation}\label{invert}
d_{\widehat{\cal D\cal G}}([a,u]_2,[a^{-1},u^{-1}]_1)\leq 2.\end{equation}

We first claim that 
\begin{equation}\label{step1}
d_{\widehat{\cal D\cal G}}(\psi_{b_1^{k_1}\cdots b_{n-1}^{k_{n-1}}b_n^{k_n}}(D),
\psi_{b_1^{k_1}\cdots b_{n-1}^{k_{n-1}}b_n^{\ell_n}}(D))\leq 
6\vert \ell_{n}-k_{n}\vert +8.\end{equation}

Namely, the estimate (\ref{distanceestimate})
and the fact that point-pushing of $p_2$ induces an isometry for $d_{\widehat{\cal D\cal G}}$ on
${\cal D\cal G}$ imply that
\[d_{\widehat{\cal D\cal G}}(\psi_{b_n^{k_n}}(D),\psi_{b_n^{\ell_n}}(D))\leq 6\vert \ell_n-k_n\vert.\]

We now use the inequality (\ref{crucial2}). 
As $\psi_{b_n^{u}}(D)=[a\cdot b_n^u]_2$ for all $u$, 
the estimate (\ref{invert}) shows that 
\[d_{\widehat{\cal D\cal G}}([a^{-1}\cdot b_n^{-k_n}]_1,[a^{-1}\cdot b_n^{-\ell_n}]_1)
  \leq 6\vert k_n-\ell_n\vert +4.\]
Apply to both disks $[a^{-1}\cdot b_n^{-k_n}]_1,[a^{-1}\cdot b_n^{-\ell_n}]_1$
point-pushing
of the point $p_1$ along a loop based at $p_1$ 
representing the homotopy class
$\Psi(b_{n-1}^{-k_{n-1}}\cdots b_1^{-k_1})$.
As point-pushing induces an isometry   
on the disk graph (and composition is read from left to right), we obtain
\[d_{\widehat{\cal D\cal G}}([a^{-1}\cdot b_n^{-k_n}b_{n-1}^{-k_{n-1}}\cdots b_1^{-k_1}]_1,
[a^{-1}\cdot b_n^{-\ell_n}b_{n-1}^{-k_{n-1}}\cdots  b_1^{-k_1}]_1)\leq 6\vert k_n-\ell_n\vert +4.\]
Using again the estimate (\ref{invert}), this yields the estimate 
(\ref{step1}) we
wanted to show.

Point-pushing of the point $p_2$ along the loop $b_n^{-\ell_n}$ 
transforms the \emph{pair} of disks 
$(\psi_{b_1^{k_1}
  \cdots b_{n-1}^{k_{n-1}}b_n^{\ell_n}}(D),D)$ to the pair
$(\psi_{b_1^{k_1}\cdots b_{n-1}^{k_{n-1}}}D,\psi_{b_n^{-\ell_n}}(D))$. 
As point-pushing acts as an isometry for $d_{\widehat{\cal D\cal G}}$, 
we view this operation as a change of basepoints which does not
change distances.

In a second step, we use the reasoning which led to the estimate
(\ref{step1}) to deduce that
\begin{equation}
  d_{\widehat{\cal D\cal G}}(\psi_{b_1^{k_1}\cdots b_{n-2}^{k_{n-2}}b_{n-1}^{k_{n-1}}}(D),
\psi_{b_1^{k_1}\cdots b_{n-2}^{k_{n-2}}b_{n-1}^{\ell_{n-1}}}(D))\leq 
6\vert \ell_{n-1}-k_{n-1}\vert +8.\notag\end{equation}
As a next step, we change the basepoint again. Using point-pushing
of the point $p_2$ along the loop 
$b_2^{-\ell_{n-1}}$, the pair
$(\psi_{b_1^{k_1}\cdots b_{n-2}^{k_{n-2}}b_{n-1}^{\ell_{n-1}}}(D),\psi_{b_n^{-\ell_n}}D)$
transforms to the pair 
\[(\psi_{b_1^{k_1}\cdots b_{n-2}^{k_{n-2}}}(D),\psi_{b_n^{-\ell_n}b_{n-1}^{-\ell_{n-1}}}D).\] 
Proceeding inductively, in $n$ steps we transform the pair
$(\psi_{b_1^{k_1}b_2^{k_2}\cdots b_n^{k_n}}(D),D)$ to the pair
$(D,\psi_{b_n^{-\ell_n}\cdots b_1^{-\ell_1}}(D))$, 
changing the value of the function $d_{\widehat{\cal D\cal G}}$ 
 by at most $\sum_i(6\vert \ell_i-k_i\vert +8)$.

Now apply one last time point-pushing of the point $p_2$ along the loop  
$b_1^{\ell_1}\cdots b_n^{\ell_n}$ to the pair 
\[(D,\psi_{b_n^{-\ell_n}\cdots b_1^{-\ell_1}}(D))\] and 
obtain the pair 
$(\psi_{b_1^{\ell_1}\cdots b_n^{\ell_n}}(D),D)$.
Using again that point pushing preserves $d_{\widehat{\cal D\cal G}}$, 
we conclude that
\[d_{\cal S\cal G}(\Pi\Lambda(k_1,\dots,k_n),\Pi\Lambda(\ell_1,\dots,\ell_n))\leq 
\sum_i(6\vert \ell_i-k_i\vert +8)\] 
as claimed.

Now Lemma 4.15 of \cite{SS14} 
and the discussion on the bottom of p.592 and on the top of
p.594 in \cite{SS14} shows 
that there is a number $c>0$ such that
\begin{equation}\label{upper1}\sum_{i=1}^n\vert k_i-\ell_i\vert \leq 
c\vert b_n^{-k_n}b_{n-1}^{-k_{n-1}}\cdots 
b_1^{-k_1}b_1^{\ell_1}b_2^{\ell_2}\cdots b_n^{\ell_n}
\vert_{g+1}^{simple}.\end{equation}

We give a short summary of the proof of this fact as 
found in \cite{SS14}. 
Namely, 
following Definition 4.9 of \cite{SS14}, we say that 
a word $w$ in the letters ${\cal A}\cup {\cal A}^{-1}$ has 
\emph{conjugate reduced length at most $k$} if there exist freely reduced
words $v_1,\dots,v_\ell,u_1,\dots,u_\ell$ such that.
\begin{enumerate}
\item[(a)] $w=v_1^{u_1}v_2^{u_2}\cdots v_\ell^{u_\ell}$, where $v_j^{u_j}=u_j^{-1}v_ju_j$, and
\item[(b)] $k=(\ell-1)+\vert v_1\vert_{g+1}^{simple}+\cdots +\vert v_\ell\vert_{g+1}^{simple}$.
\end{enumerate}
The number $k$ is called the \emph{conjugate reduced 
$g+1$-length associated to the
decomposition}. 
The minimal number $k$ for which such a decomposition exists is called  
the \emph{conjugate reduced length of $w$}, and it is denoted by
$\vert w\vert^{cr}$. 

The easy Lemma 4.15 of \cite{SS14} states that $\vert w\vert_{g+1}^{simple}\geq 
\vert w\vert^{cr}$, so it suffices to estimate $\vert w\vert^{cr}$ from below
for 
$w=  b_n^{-k_n}b_{n-1}^{-k_{n-1}}\cdots 
b_1^{-k_1}b_1^{\ell_1}b_2^{\ell_2}\cdots b_n^{\ell_n}$. 

Definition 4.10 of \cite{SS14} is geared to this end. A \emph{cancelling pair}
in the reduced word $w$ is a pair of subwords of the form $u,u^{-1}$. 
A \emph{nested family ${\cal F}$ 
of cancelling pairs} is a finite collection of disjoint cancelling pairs so that
if $v,v^{-1}\in {\cal F}$ and $u,u^{-1}\in {\cal F}$ then $v$ occurs between
$u,u^{-1}$ if and only if this is true for $v^{-1}$. For such a family ${\cal F}$ 
of cancelling pairs let $w-{\cal F}$ be the finite collection of subwords of 
$w$ obtained by erasing the words from ${\cal F}$. Define
\[ \vert w-{\cal F}\vert_{g+1}^{simple}=\vert {\cal F}\vert +
\sum_{w^\prime \in w-{\cal F}}\vert w^\prime \vert_{g+1}^{simple}.\]

 The required estimate follows from Lemma 4.11 of \cite{SS14} which states that
 \begin{equation}\label{lemma411}
 \vert w\vert^{cr}\geq \min_{\cal F}(\max \bigl\{ \frac{\vert {\cal F}\vert }{2}-1,
 \frac{1}{5}\vert w-{\cal F}\vert_{g+1}^{simple}-3\bigr\}).\end{equation}
  
To apply this estimate to the above word 
$w$, let ${\cal F}$ be a nested family of cancelling pairs for $w$ which minimizes
the expression on the right hand side of equation (\ref{lemma411}) and write
$d=\sum\vert k_i-\ell_i\vert$. 
If $\vert {\cal F}\vert \geq d/10$ then we immediately obtain the required estimate.
Otherwise note that by removing a 
cancelling pair we can at most delete a subword 
of a string of the form $b_i^{\min\{k_i,\ell_i\}}$. 
Furthermore it is easy to see that
$\vert b_t^s\vert_{g+1}^{simple}\geq \vert s\vert $ for all $s$. Thus if 
$\vert {\cal F}\vert \leq d/10$ then a rough counting of the simple norm of the
subsegments of $w-{\cal F}$ as carried out in detail on p.593 of \cite{SS14} 
yields again the required estimate.

On the other hand,
by Lemma \ref{simplelength}, we have  
\begin{equation}\label{upper}
d_{\cal S\cal G}(\Pi\Lambda(k_1,\dots,k_n),\Pi\Lambda(\ell_1,\dots,\ell_n))
\geq \frac{1}{2}
 \vert b_n^{-k_n}b_{n-1}^{-k_{n-1}}\cdots 
b_1^{-k_1}b_1^{\ell_1}b_2^{\ell_2}\cdots b_n^{\ell_n}
\vert_{g+1}^{simple}.\end{equation}
The estimates (\ref{lower}), (\ref{upper1}) and (\ref{upper}) together
show that the distance in 
${\cal S\cal G}$ of the images of $\Pi(D)$ under the 
point pushing of $p_2$ along $b_1^{k_1}b_2^{k_2}\cdots b_n^{k_n}$
and by $b_1^{\ell_1}b_2^{\ell_2}\cdots b_n^{\ell_n}$ 
 is bounded from above and below by a fixed positive 
multiple of $\sum_{i=1}^n \vert k_i-\ell_i\vert$.
Thus the map $\Lambda:\mathbb{Z}^n\to {\cal S\cal G}$ is 
a quasi-isometric embedding. 
The theorem in the case that $g$ is even follows,

The argument can be adjusted to the case that $g$ is odd 
as follows.
Let $H_0$ be a handlebody of odd genus $g\geq 5$. Choose a non-separating 
$I$-bundle generator $c$. Then $H_0$ is the oriented $I$-bundle over 
a non-orientable surface $F$ with connected boundary $\partial F=c$. The surface
$F$ can be obtained from an orientable surface $F_0$ of genus $(g-1)/2\geq 2$ 
whose boundary
consists of 2 connected components $c_0,c_1$ by attaching a M\"obius band to $c_1$. 
The orientation cover of $F$ equals the complement in $\partial H_0$ of an open annulus
with core curve $c$, and the preimage of $F_0$ consists of two copies of $F_0$ which are glued
along an annulus. The fundamental group of $F_0$ is a free group in $g$ generators, and the 
inclusion of a component of its preimage in $\partial H_0$ into $H_0$ defines an isomorphism
on fundamental groups.

The argument in the beginning of this proof now applies verbatim using the surface
$F_0$ instead of $F$ and noting that 
we may choose disjoint generating arcs for the fundamental group of $F_0\subset F$ with endpoints on the
boundary of $F$ with the property that there is a partition of $\partial F$ into 
$(g-1)/2+1\geq 2$ disjoint intervals, each containing the endpoints of one or two arcs. This suffices to control
the distance in the disk graph of a disk obtained from the base disk $D$ by point pushing $p_2$ 
along a loop defined by the word $b_t$ in the corresponding generators. The rest of the argument
is identical to the argument for connected sums of an even number $g\geq 4$ of copies of $S^1\times S^2$ 
with two spots. 
\end{proof}

\bigskip\bigskip

\noindent
MATH. INSTITUT DER UNIVERSIT\"AT BONN, ENDENICHER ALLEE 60, 
53115 BONN, GERMANY\\
\bigskip\noindent
e-mail: ursula@math.uni-bonn.de

\end{document}